\documentclass{amsart}
\usepackage{amsmath,amssymb}
\makeatletter
\DeclareRobustCommand{\lunderset}[3][1pt]{%
  \binrel@{#3}%
  \binrel@@{%
    \vbox{%
      \offinterlineskip\m@th
      \ialign{##\hfil\cr$\displaystyle#3$\cr\noalign{\vskip#1}$\scriptstyle#2$\cr}%
    }%
  }%
}
\makeatother
\usepackage{tikz-cd}
\usepackage{amsmath}
\usepackage{amsfonts}
\usepackage{amssymb,amsthm}

\def\scal#1#2{\langle #1\bv#2 \rangle}

\def\shuffle{\mathop{_{^{\sqcup\!\sqcup}}}}

\def\ncp#1#2{#1\langle #2\rangle}
\def\ncs#1#2{#1\langle \!\langle #2\rangle \!\rangle}

\newcommand{\calB}{{\mathcal B}}
\newcommand{\calC}{{\mathcal C}}
\newcommand{\calD}{{\mathcal D}}

\newcommand{\calG}{{\mathcal G}}
\newcommand{\calH}{{\mathcal H}}

\newcommand{\calL}{{\mathcal L}}
\newcommand{\calM}{{\mathcal M}}

\newcommand{\calP}{{\mathcal P}}

\newcommand{\calS}{{\mathcal S}}

\newcommand{\calU}{{\mathcal U}}

\newcommand{\calX}{{\mathcal X}}

\newcommand{\calZ}{{\mathcal Z}}

\newcommand{\N}{{\mathbb N}}
\renewcommand{\P}{\N_{>0}}
\newcommand{\Z}{{\mathbb Z}}
\newcommand{\Q}{{\mathbb Q}}
\newcommand{\R}{{\mathbb R}}
\newcommand{\C}{{\mathbb C}}

\newcommand{\Sum}[2]{\displaystyle{\sum_{#1}^{#2}}}

\def\pol#1{\langle #1 \rangle}
\def\Lyn{{\mathcal Lyn}}

\def\CX{\C \langle X \rangle}
\def\CY{\C \langle Y \rangle}
\def\abs#1{|#1|}

 \def\shuffle{\mathop{_{^{\sqcup\!\sqcup}}}} 

\gdef\stuffle{\;%
  \setlength{\unitlength}{0.0125cm}%
  \begin{picture}(20,10)(220,580) 
  \thinlines 
  \put(220,592){\line( 0,-1){ 10}} 
  \put(220,582){\line( 1, 0){ 20}} 
  \put(240,582){\line( 0, 1){ 10}} 
  \put(230,592){\line( 0,-1){ 10}} 
  \put(225,587){\line( 1, 0){ 10}} 
  \end{picture}\; 
}

\newtheorem{corollary}{Corollary}
\newtheorem{proposition}{Proposition}
\newtheorem{theorem}{Theorem}
\newtheorem{lemma}{Lemma}
\newtheorem{definition}{Definition}

\newtheorem{example}{Example}

\newtheorem{remark}{Remark}

\newcommand{\Li}{\operatorname{Li}}

\def\L{\mathrm{L}}
\def\H{\mathrm{H}}

\def\P{\mathrm{P}}

\def\deg{\mathrm{deg}}

\newcommand{\poly}[2]{#1 \langle #2 \rangle}

\def\QX{\poly{\Q}{X}}

\def\CX{\poly{\C}{X}}

\def\CY{\poly{\C}{Y}}

\newcommand{\serie}[2]{#1 \langle \! \langle #2 \rangle \! \rangle}

\def\pol#1{\langle #1 \rangle}

\def\AX{A \langle X \rangle}

\newcommand*{\longtwoheadrightarrow}{\ensuremath{\joinrel\relbar\joinrel\twoheadrightarrow}}

\def\AX{\ncp{A}{X}}

\def\AY{\ncp{A}{Y}}

\def\CX{\ncp{\C}{X}}

\def\Sum{\displaystyle\sum}

\def\path{\rightsquigarrow}

\def\bv{\mid}

\def\abs#1{\bv\!#1\!\bv}

\def\trr{\triangleright}
\def\trl{\triangleleft}
\def\conc{\tt{conc}} 
\def\span{\mathrm{span}} 
\def\supp{\mathrm{supp}} 
\def\rd{\triangleright}
\def\rg{\triangleleft}
\def\deg{\mathop\mathrm{deg}\nolimits}

\def\supp{\mathop\mathrm{supp}\nolimits}

\def\dbinom#1#2{{#1\choose#2}}
\def\binom#1#2{{#1\choose#2}}

\gdef\ministuffle{{\scriptstyle \stuffle}}

\begin{document}
\title[Representative Functions on Free Monoids]{Various Bialgebras of Representative Functions on Free Monoids}
\author{V. Hoang Ngoc Minh}
\address{Universit\'e de Lille, 1 Place D\'eliot, 59024 Lille, France.}
\curraddr{Universit\'e de Lille, 1 Place D\'eliot, 59024 Lille, France.}
\email{vincel.hoang-ngoc-minh@univ-lille.fr}
\thanks{}

\date{}
\dedicatory{}

\begin{abstract}
Factorization and decomposition of {representative} functions on a free monoid $\calX^*$ (generated by an alphabet $\calX$) and with values in a ring $A$ containing $\Q$ are equivalent to factorization and decomposition of their graphs (within the $A$-algebra of rational noncommutative series over $\calX$) admitting {linear representations} (thanks to the Kleene-Sch\"utzenberger theorem).

To factorize and to decompose effectively these graphs, we examine various products of noncommutative series (as concatenation, shuffle and its $\phi$-deformations) and co-products such that, for $A$ is a field $K$, their associated non graded commutative and co-noncommutative bialgebras of series are isomorphic to the Sweedler's dual of the graded noncommutative co-commutative bialgebras of polynomials having, for the concatenation, only Kleene stars of the planes as characters, or equivalently, only the planes are infinitesimal characters (thanks to a Ree's theorem like).
\end{abstract}

\maketitle
\tableofcontents

\section{Introduction}\label{intro}
Hopf algebras appeared in algebraic geometry and topology, theory of algebraic groups and representation theory, in which representative functions, with values in a field $K$, were investigated systematically on a group $G$ \cite{abe,Cartier2,Hochschild}.
As in \cite{CartierPatras}, let\begin{eqnarray}\label{f}
f:\calX^*\longrightarrow K.
\end{eqnarray}
be a function with $K$-values and on the monoid $\calX^*$ generated by the alphabet $\calX$.

\begin{remark}\label{multiindex}
Multiindex $(s_1,\ldots,s_r)\in\N_{\ge1}^r$ belongs to the monoid generated by $\N_{\ge1}$, denoted by $(\N_{\ge1})^*$.
Multiindexes one-to-one corresponds\footnote{See also the morphisms of algebras $\pi_Y$ and $\pi_X$ in Appendix below.} to the words $y_{s_1}\cdots y_{s_r}$ (resp. $x_0^{s_1-1}x_1\ldots x_0^{s_r-1}x_1$) of the monoid $Y^*$ (resp. $X^*$) generated by the alphabet $Y=\{y_k\}_{k\ge1}$ (resp. $X=\{x_0,x_1\}$).

Multiindexed combinatorial objets as quasi-symetric functions, diagrams and similar ones are also functions on $(\N_{\ge1})^*$.
By composition of functions or by bijections, functions on monoids of theses objets are also functions on $(\N_{\ge1})^*$.
\end{remark}

\begin{example}\label{pos}
The functions $\Li_{\bullet}$ and $\H_{\bullet}$ on the monoid $(\N_{\ge1})^*$ map the multiindex $(s_1,\ldots,s_r)$, respectively, to the complex function $\Li_{s_1,\ldots,s_r}$ and to the arithmetical function $\H_{s_1,\ldots,s_r}$. These are defined by
\begin{eqnarray*}
\Li_{s_1,\ldots,s_r}(z):=\sum_{0<n_1<\cdots<n_r\phantom{\le n}}\frac{z^{n_1}}{n_1^{s_1}\cdots n_r^{s_r}},\cr
\H_{s_1,\ldots,s_r}(n):=\sum_{0<n_1<\cdots<n_r\le n}\frac1{n_1^{s_1}\cdots n_r^{s_r}}
\end{eqnarray*}
and are so-called polylogarithm and harmonic sum, respectively. These lie as follows
\begin{eqnarray*}
\frac{1}{1-z}\Li_{s_1,\ldots,s_r}(z)=\sum_{n\ge0}\H_{s_1,\ldots,s_r}(n)z^n.
\end{eqnarray*}

By Remark \ref{multiindex} (see also Appendix below), $\Li_{\bullet}$ (resp. $\H_{\bullet}$) is also a function on the monoid $X^*$ (resp. $Y^*$):
\begin{eqnarray*}
\Li_{x_0^{s_1-1}x_1\ldots x_0^{s_r-1}x_1}=\Li_{s_1,\ldots,s_r}&\mbox{and}&\H_{y_{s_1}\ldots y_{s_r}}=\H_{s_1,\ldots,s_r}.
\end{eqnarray*}
\end{example}

\begin{example}\label{neg}
Similarly to Example \ref{pos}, one also defines, for any multiindex $(s_1,\ldots,s_r)$ of $(\N)^*$ associated with the word $y_{s_1}\ldots y_{s_r}$ in $Y_0^*$, where $Y_0:=\{y_k\}_{k\ge0}$,
\begin{eqnarray*}
\Li_{-s_1,\ldots,-s_r}(z):=\sum_{0<n_1<\cdots<n_r}{n_1^{s_1}\cdots n_r^{s_r}}z^{n_1},\\
\H_{-s_1,\ldots,-s_r}(n):=\sum_{0<n_1<\cdots<n_r\le n}{n_1^{s_1}\cdots n_r^{s_r}}.
\end{eqnarray*}
These also lie as follows
\begin{eqnarray*}
\frac{1}{1-z}\Li_{-s_1,\ldots,-s_r}(z)=\sum_{n\ge0}\H_{-s_1,\ldots,-s_r}(n)z^n.
\end{eqnarray*}
\end{example}

\begin{example}\label{MZV}
With the notations in Example \ref{pos}, for $s_1>1$, the following limits exist and coincide (see a theorem by Abel):
\begin{eqnarray*}
\zeta(s_1,\ldots,s_r)=\left\{\begin{array}{rcl}
\zeta(x_0^{s_1-1}x_1\ldots x_0^{s_r-1}x_1)&:=&\lim\limits_{z\to1}\Li_{x_0^{s_1-1}x_1\ldots x_0^{s_r-1}x_1}(z),\\
\zeta(y_{s_1}\ldots y_{s_r})&:=&\lim\limits_{n\to+\infty}\H_{y_{s_1}\ldots y_{s_r}}(n).
\end{array}\right.
\end{eqnarray*}

This common limit is so-called \textit{polyzeta}\footnote{Polyzeta is a contraction of polymorphism and zeta.} and is obtained as image by the real $\zeta$-\textit{polymorphism} of the sequence $(s_1,\ldots,s_r)$ belonging to the monoid $(\N_{\ge1})^*$ or equivalently, by Remark \ref{multiindex}, it is also images of the sequences $y_{s_1}\ldots y_{s_r}\in Y^*$ or $x_0^{s_1-1}x_1\ldots x_0^{s_r-1}x_1\in X^*$ (see Appendix below):
\begin{eqnarray*}
\zeta(y_{s_1}\ldots y_{s_r})
=\zeta(x_0^{s_1-1}x_1\ldots x_0^{s_r-1}x_1)
=\sum_{0<n_1<\cdots<n_r}\frac1{n_1^{s_1}\cdots n_r^{s_r}}.
\end{eqnarray*}
\end{example}

The function $f$ in \eqref{f} is a \textit{representative} if and only if there is finitely many functions $\{f'_i,f''_i\}_{i\in I_{finite}}$ of $K^{\calX^*}$, which can be choosen to be \textit{representative} functions such that, for any $u$ and $v\in\calX^*$, one has \cite{CartierPatras}
\begin{eqnarray}\label{representative}
f(uv)=\sum_{i\in I_{finite}}f'_i(u)f''_i(v).
\end{eqnarray}

With the notations in \eqref{f}--\eqref{representative}, the coproduct of \textit{representative} function $f$ can be defined in duality with the product in $\calX^*$ (\textit{i.e.} the concatenation, denoted by $\tt conc$ and omitted when there is no ambiguity) as follows \cite{CartierPatras}
\begin{eqnarray}\label{Delta_f}
\forall u,v\in\calX^*,
&\Delta(f)(u\otimes v)=f(uv),
&\Delta(f)=\sum_{i\in I_{finite}}f'_i\otimes f''_i.
\end{eqnarray}
Moreover, the graph of $f$, viewed as a noncommutative generating series over $\calX$ and with coefficients in $K$, is described as follows
\begin{eqnarray}\label{ngs}
S=\sum_{w\in\calX^*}\scal{S}{w}w,&\mbox{where}&\scal{S}{w}=f(w).
\end{eqnarray}

By definition, any series $S$ is a function from $\calX^*$ to $K$ (see \eqref{S}--\eqref{graph} below) mapping the word $w$ in $\calX^*$ to $\scal{S}{w}$, so-called \textit{coefficient} of $w$ in $S$, and this series is \textit{rational} if and only if there is a \textit{linear representation} of dimension $n$, \textit{i.e.} the triplet $(\nu,\mu,\eta)$, where the matrices $\nu$ and $\eta$ belong to $M_{1,n}(K)$ and $M_{n,1}(K)$, respectively, and the morphism of monoids $\mu:\calX^*\longrightarrow M_{n,n}(K)$ (a \textit{linear representation} of $\calX^*$) such that (Kleene-Sch\"utzenberger theorem \cite{berstel}, see also Theorem \ref{KS} below)
\begin{eqnarray}\label{tripletcoeff}
\forall w\in\calX^*,&\scal{S}{w}=\nu\mu(w)\eta.
\end{eqnarray}
By left and right shifts (see Definition \ref{dec1} below), minimization algorithms provide minimal linear representations of the smallest dimension \cite{berstel}.

With the triplet $(\nu,\mu,\eta)$, for any $i=1,\ldots,n$, let $G_i$ and $D_i$ be the \textit{rational} series, over $\calX$ with coeffiocients in $K$, admitting respectively $(\nu,\mu,e_i)$ and $({}^te_i,\mu,\eta)$ as \textit{linear representations} of dimension $n$, where (see Proposition \ref{PQ} below)
\begin{eqnarray}\label{sweedler}
\lunderset{\hskip63.5mm\uparrow i}{\begin{array}{rcl}
e_i\in\calM_{1,n}(A)&\mbox{and}&{}^te_i=\begin{matrix}(0&\ldots&0&1&0&\ldots&0)\end{matrix}.
\end{array}}
\end{eqnarray}
Then, extending the coproduct $\Delta_{\conc}$ over the Cauchy $K$-algebras of polynomials and then of series (see \eqref{conc}--\eqref{Dconc} below), one has (see Proposition \ref{PQ} below) \cite{CM}
\begin{eqnarray}\label{Delta_S}
\Delta_{\conc}(S)=\sum_{1\le i\le n}G_i\otimes D_i
\end{eqnarray}

By \eqref{Delta_f}--\eqref{Delta_S}, for any $u$ and $v\in\calX^*$, note also that
\begin{eqnarray}
\Delta(f)(u\otimes v)=&f(uv)&=\scal{S}{uv},\\
\scal{\Delta(S)}{u\otimes v}
=&\displaystyle\sum_{1\le i\le n}\scal{G_i}{u}\scal{D_i}{v}&=\sum_{1\le i\le n}f'_i(u)f''_i(v).\label{Delta_fi}
\end{eqnarray}
Moreover,
$f'_i(u)=\scal{G_i}{u}=\nu\mu(u)e_i$ and $f''_i(v)=\scal{D_i}{v}={}^te_i\mu(v)\eta$ ($1\le i\le n$).

Hence, the function $f$ on a monoid $\calX^*$ with values in a field $K$, in \eqref{f}, is {representative} if and only if the series $S$ over an alphabet $\calX$ with coefficients in $K$, in \eqref{S}, is {rational} and $S$ is also said to be {representative} (so do $G_i$ and $D_i$) \cite{orlando,galoisdiff}.

In the present work, representative series is viewed as noncommutative generating series of representative functions, on a monoid $\calX^*$ with values in a ring $A$ containing $\Q$, and are factorized and decomposed within their associated $A$-bialgebras.

For that, effective factorization and decomposition will be done, as in \eqref{f}--\eqref{Delta_fi} and in continuation with our previous works in \cite{PMB,QTS12,ACA,hoangjacoboussous,FPSAC96,CM}. In particular, on the one hand, the monoidal factorizations (by Lazard and by Sch\"ut\-zenberger) \cite{lothaire,viennotgerard} will be used. On the other hand, the results already obtained over $\Q$ or $\C$ in \cite{PMB,Linz,ACA,CM} (dealing with shuffle product, $\shuffle$, and quasi-shuffle product, $\stuffle$, of noncommutative series, over an alphabet $\calX$ and with coefficients in a ring $A$ containing $\Q$, and leading to the studies\footnote{See Appendix below for an synthesis of theses studies.} of polylogarithms and harmonic sums) will be also extended.

The organization of the paper is following
\begin{itemize}
\item In Section \ref{FPS}, we will examine combinatorials aspects of various products (as concatenation, shuffle and its $\phi$-deformations (denoted by $\shuffle_{\phi}$)  and their coproducts, for which primitive and grouplike series will be characterized, by Propositions \ref{proposition1}--\ref{proposition2}. Moreover, pairs of dual bases for $\shuffle$ and for ${\shuffle}_{\phi}$ $A$-bialgebras will be constructed, by Theorems \ref{CQMM}--\ref{isomorphy}, to factorize diagonal series and then grouplike series (see Corollary \ref{factorisations} below).

\item In Section \ref{representativeseries}, representative series will be factorized and decomposed (see Corollary \ref{factorized} below). For $A=K$, the Sweedler's duals of the $\shuffle$ (resp. ${\shuffle}_{\phi}$) bialgebras of polynomials will proved to be isomorphic to (\textit{non graded}) bialgebras of representative series (see Propositions \ref{PQ}--\ref{simple}, Theorem \ref{KS}, Corollaries \ref{Sweedler}--\ref{star} below). Secondly, for $\Delta_{\conc}$, only Kleene stars of the plane are grouplike and only planes are primitive (see Definition \ref{cylindric}, Theorem \ref{exchangeable} and Corollary \ref{Kleene} below).
\end{itemize}

Historically, noncommutative formal series were introduced for the first time by Sch\"utzenberger \cite{schutz2} to study problems related to theoretical computer science, such as language theory and the theory of automata. He generalized a theorem by Kleene \cite{kleene} to noncommutative formal series \cite{schutz3}. He then showed the fundamental role played, for the study of noncommutative formal series, by \textit{linear representations} of free monoids. In particular irreducible representations, allowed him to find fine results on the growth of the coefficients in \eqref{tripletcoeff} \cite{schutz4}.

After Sch\"utzenberger, we must cite the works of Fliess \cite{fliess-these}, Jacob \cite{jacob-these} and Reutenauer \cite{reutenauer-these}.
They developed sets of fundamental tools for combinatorial studies of free monoids, linked to the theory of automata. Indeed, using the Hankel matrices, Fliess characterized noncommutative rational series (a series is rational if and only if the rank of its Hankel matrix is finite) and series with positive coefficients. Using matrix representations, Jacob generalized the notion of loop in finite automata (star lemma) which enables to solve problems of decidability of the finiteness of the coefficients of rational series. Reutenauer characterized rational series by their syntactic algebra (a series is rational if and only if the dimension of its syntactic algebra is finite) and defined the notion of varieties of formal series in the sense of Eilenberg \cite{eilenberg}. These works were widely reproduced in books by Berstel and Reutenauer \cite{berstel}, by Kuich and Salomaa \cite{kuichsalomaa} and by Salomaa and Soittola \cite{SalomaaSoittola}.

As Chomsky and Sch\"utzenberger showed that the algebraic languages are the supports of algebraic series \cite{chomskyschutz} as being solutions of a system of propre algebraic equations with integer coefficients. Nivat \cite{nivat} and then Fliess \cite{fliess-these} showed that the study of noncommutative algebraic series mainly depends on the study of rational transductions of algebraic series. Introducing the notion of regulated transduction (to establish the converse of the Shamir's theorem), Jacob showed that the image by regulated rational (resp. algebraic) transduction of a rational (resp. algebraic) series is a rational (resp. algebraic) series, opening a way to the study of nice families of noncommutative series. Kuich further developed this method in an approach by cycle-free push-down automata to study noncommutative algebraic series \cite{kuich}.

Noncommutative series benefited then from knowledge of language and automata theories. In return, techniques and results in \cite{fliess-these,jacob-these,reutenauer-these} enabled new visions in these fields \cite{berstel}.
These developments made the algebra of formal series a preferred tool for the syntactic study of operator algebras \cite{reutenauer}. Furthermore, algebra of formal series proved to be a particularly well-suited tool for implementations of effective calculations in computer algebra systems \cite{jacobreutenauer,jacobIMACS}.
In particular, since the input-output behaviors of (nonlinear) dynamical systems (or causal functionals) was encoded  by noncommutative series (see \cite{jacob} for an introduction), the noncommutative symbolic computation (a generalization of the Heaviside's calculus) became efficient for dealing with special functions (hypergeometric functions, hyperlogarithms, polylogarithms \cite{ACA,FPSAC96,livre,CM}) in the study of (nonlinear) differential equations in
control theory \cite{fliess1,fliess2,hoangjacoboussous,jacob,reutenauerrealisation}
and in quantum electrodynamics (QED) \cite{PMB,QTS12,ACA,FPSAC96}.

Let us illustrate our purposes by considering the following dynamical system
\begin{eqnarray}\label{nonlinear}
\left\{\begin{array}{rcl}
\displaystyle\frac{d}{dz}q(z)&=&A_0(q)u_0(z)+\ldots+A_m(q)u_m(z),\\
q(z_0)&=&\eta,\\
y(z)&=&h(q(z)),
\end{array}\right.
\end{eqnarray}
where $y$ is the output and
\begin{itemize}
\item $(u_i)_{0\le i\le m}$ are the inputs (and their inverses $(u_i^{-1})_{0\le i\le m}$) belong  to a subring of the differential ring of holomorphic functions $\calH(\Omega)$ with the neutral element $1_{\calH(\Omega)}$ over the simply connected manifold $\Omega$.

\item $h$ is the observation defined within a fixed connected neighbourhood\footnote{Here, the points are loosely identified with their coordinates through some chart $\varphi_U:U\to \C^n$ likewise, in \cite{reutenauerrealisation}, the space of holomorphic functions $\calH(U)$ is described by $\C^{\rm cv}[\![q_1,\ldots,q_n]\!]$.} $U$ of of the initial state $\eta$,

\item $(A_i)_{0\le i\le m}$ are the vector fields, defined with respect to the coordinates, by
\begin{eqnarray}\label{vectorfield}
A_i=\sum_{1\le j\le n}A_i^j(q)\frac{\partial}{\partial q_j},&\mbox{with}&A_i^j(q)\in\calH(U),
\end{eqnarray}
\end{itemize}

It is convenient (and possible) to separate the contribution of the vector fields $(A_i)_{0\le i\le m}$ and that of the differential forms $(\omega_i)_{0\le i\le m}$ defined by the inputs, \textit{i.e.} $\omega_i=u_idz$, through the encoding alphabet $X=\{x_i\}_{0\le i\le m}$ generating the monoid $(X^*,1_{X^*})$. Indeed, the output $y$ (depending on $z_0$) can be computed by
\begin{eqnarray}\label{output}
y(z)=\sum_{w\in X^*}\alpha_{z_0}^z(w)A_wh_{|_{\eta}},
\end{eqnarray}
where the differential operator $A_w$ and  the iterated integral $\alpha_{z_0}^z(w)$ (of the differential forms $(\omega_i)_{0\le i\le m}$ and along the path $z_0\path z$ over the manifold $\Omega$) are defined by $\alpha_{z_0}^z(1_{X^*})=1_{\calH(\Omega)}$ and $A_{1_{X^*}}=\mathrm{Id}$ and recursively by \begin{eqnarray}\label{notation}
\forall x_i\in\calX,\forall v\in\calX^*,&A_{x_iv}=A_i\circ A_v,&
\alpha_{z_0}^z(x_iv)=\int_{z_0}^z\omega_i(s)\alpha_{z_0}^{s}(v).
\end{eqnarray}
In \eqref{output}, the output $y$ is viewed as the pairing, under convergence conditions \cite{ACA,fliess2,reutenauerrealisation}, between the Chen series \cite{chen,Hain} and the generating series of \eqref{nonlinear} \cite{fliess2}, defined as noncommutative series over the alphabet $X$, respectively as follow
(in which $\alpha_{z_0}^z(w)$ and $A_wh$ belong to the differential rings $\calH(\Omega)$ and $\calH(U)$, respectively)
\begin{eqnarray}\label{ngsofds}
C_{z_0\path z}:=\sum_{w\in X^*}\alpha_{z_0}^z(w)w
&\mbox{and}&\sigma h_{|_{\eta}}:=\sum_{w\in X^*}A_wh_{|_{\eta}}w.
\end{eqnarray}

\begin{remark}
In control theory, the inputs $\{u_i\}_{1\le i\le m}$ in \eqref{nonlinear} are encoded by the alphabet $X=\{x_i\}_{1\le i\le m}$ and are Lebesgue integrable real-valued functions on the interval $[0,T]$ ($T\in\R_{\ge0}$, is so-called the duration of the controls) and then the Chen series, of $\{\omega_i\}_{1\le i\le m}$ and over the interval $[0,T]$, belongs to $\ncs{L^{\infty}([0,T],\R)}{X}$.
\end{remark}

Now, for simplification, let us focus on the representative (or rational) series as generating series of the following type of dynamical systems \cite{fliess1,fliess2,jacob}
\begin{eqnarray}\label{linear}
dq(z)=\big(M_0\omega_0(z)+\ldots+M_m\omega_m(z)\big)q(z),&
q(z_0)=\eta,&
y(z)=\nu q(z),
\end{eqnarray}
where $\{M_i\}_{1\le i\le n}$ and $\nu$ are, respectively, $n\times n$ and $1\times n$ matrices with coefficients in the ring $A$. In this case, the vectors fields \eqref{vectorfield} can be expressed as follows
\begin{eqnarray}
A_i=\sum_{1\le j\le n}\Big(\sum_{1\le k\le n}m_{k,j}q_k\Big)\frac{\partial}{\partial q_j}
\end{eqnarray}
and, using the morphism $\mu$ from $X^*$ to the ring $\calM_{n,n}(A)$ mapping $x_i$ to $M_i$, the series $\sigma h_{|_{\eta}}$ in \eqref{output} admits $(\nu,\mu,\eta)$ as linear representation of dimension $n$ such that
\begin{eqnarray}
\forall w\in X^*,&
&\scal{\sigma h_{|_{\eta}}}{w}=A_wh_{|_{\eta}}=\nu\mu(w)\eta.
\end{eqnarray}

\begin{example}[Hypergeometric equation, \cite{ACA}]\label{Hypergeometric}
Let $t_0,t_1,t_2$ be parameters and
\begin{eqnarray*}
z(1-z)\frac{d^2}{dz^2}y(z)+[t_2-(t_0+t_1+1)z]\frac{d}{dz}y(z)-t_0t_1 y(z)=0.
\end{eqnarray*}
Let $q_1(z)=-y(z)$ and $q_2(z)=(1-z){d}/{dz}\;y(z)$. One has
\begin{eqnarray*}
\frac{d}{dz}\begin{pmatrix}
q_1(z)\cr q_2(z)\end{pmatrix}
=(M_0u_0(z)+M_1u_1(z))
\begin{pmatrix}q_1(z)\cr q_2(z)\end{pmatrix},
\end{eqnarray*}
where the inputs $u_0(z)=z^{-1}$ and $u_1(z)=(1-z)^{-1}$ are rational functions in $\C(z)$ and $M_0$ and $M_1$ are $2\times2$ matrices with coefficients in $\C[t_0,t_1,t_2]$
\begin{eqnarray*}
M_0=-\begin{pmatrix}0&0\cr t_0t_1&t_2\end{pmatrix}&\mbox{and}&M_1=-\begin{pmatrix}0&1\cr0&t_2-t_0-t_1\end{pmatrix}.
\end{eqnarray*}
Or equivalently, with $\omega_0(z)=u_0(z)dz$ and $\omega_1(z)=u_1(z)dz$,
\begin{eqnarray*}
dq(z)=A_0(q)\omega_0(z)+A_1(q)\omega_1(z)&\mbox{and}&y(z)=-q_1(z),
\end{eqnarray*}
where $A_0$ and $A_1$ are the following parametrized linear vector fields
\begin{eqnarray*}
A_0=-(t_0t_1q_1+t_2q_2)\frac{\partial}{\partial q_2}
&\mbox{and}&
A_1=-q_1\frac{\partial}{\partial q_1}-(t_2-t_0-t_1)q_2\frac{\partial}{\partial q_2}
\end{eqnarray*}
acting as follows
\begin{eqnarray*}
\frac{\partial}{\partial q_1}(q)= \frac{\partial}{\partial q_1}\begin{pmatrix}q_1\cr q_2\end{pmatrix}=\begin{pmatrix} 1\cr 0\end{pmatrix}
&\mbox{and}&
\frac{\partial}{\partial q_2}(q)=\frac{\partial}{\partial q_2}\begin{pmatrix}q_1\cr q_2\end{pmatrix}=\begin{pmatrix} 0\cr 1\end{pmatrix}.
\end{eqnarray*}

Using the notations in \eqref{notation} and  Example \ref{pos}, the iterated integrals of the differential forms $\omega_0$ and $\omega_1$ and along a path $0\path z$ over $\Omega=\widetilde{\C\setminus\{0,1\}}$ lead to the polylogarithms\footnote{In Appendix, a synthesis of results on polylogarithms will be presented.}, \textit{i.e.} $\alpha_0^z(x_0^{s_1-1}x_1\ldots x_0^{s_r-1}x_1)=\Li_{x_0^{s_1-1}x_1\ldots x_0^{s_r-1}x_1}(z)=\Li_{s_1,\ldots,s_r}(z)$.
\end{example}

\section{Algebraic combinatorics on formal power series}\label{FPS}
\subsection{Various products of formal power series}\label{coproducts}
As already said in Section \ref{intro}, $A$ denotes a ring containing $\Q$ and $\calX$ denotes either a finite alphabet or an infinite alphabet, as\footnote{Note that $m$ can be infinite \cite{PMB}.

See \cite{FPSAC98,orlando} (resp. \cite{galoisdiff}), for examples with $m=1$ (resp. $m>1$).} (see also Remark \ref{orders_Rk} below)
\begin{eqnarray}\label{alphabets}
X:=\{x_i\}_{0\le i\le m}&\mbox{or}&Y:=\{y_k\}_{k\ge1},
\end{eqnarray}
generating the free monoid\footnote{$1_{\calX^*}$ is the neutral element of the monoid $\calX^*$.} $(\calX^*,1_{\calX^*})$, for the concatenation product (denoted by $\tt conc$ and omitted when there is no ambiguity).
One also denotes \cite{berstel}
\begin{eqnarray}
\calX^+:=\calX^*\calX=\calX\calX^*=\calX^*\setminus\{1_{\calX^*}\}.
\end{eqnarray}

The \textit{length} of a word $w$ in $\calX^*$ is denoted by $\abs{w}$.
For $w=y_{s_1}\cdots y_{s_r}\in Y^*$, the \textit{weight} of $w$ is the number $s_1+\cdots+s_r$ and is denoted by $(w)$.

A \textit{series} $S$ is the following map
\begin{eqnarray}\label{S}
S:\calX^*\longrightarrow A,&&w\longmapsto\scal{S}{w},
\end{eqnarray}
and its \textit{graph} is described as follows \cite{berstel}
\begin{eqnarray}\label{graph}
S=\sum\limits_{w\in\calX^*}\scal{S}{w}w.
\end{eqnarray}
The constant term of $S$ is $\scal{S}{1_{\calX^*}}$. If $\scal{S}{1_{\calX^*}}=0$ then $S$ is said to be \textit{propre}.
The image of $S$ is denoted by $\mathrm{Im}(S)$. The \textit{support} of $S$ is the following language \cite{berstel}
\begin{eqnarray}
\supp S:=\{w\in\calX^*\vert\scal{S}{w}\neq0\}.
\end{eqnarray}
If $\supp S$ is finite then it is a \textit{polynomial} of \textit{degree} defined by $\max\{\abs{w}\}_{w\in\supp S}$ \cite{berstel}.

The polynomial $P$ is \textit{homogenous} in degree $n$ if it is linear combination of words of length $n$. The \textit{characteristic series} of a \textit{language} $L\subset\calX^*$ is defined by \cite{berstel}
\begin{eqnarray}
\mathrm{char}L=\sum_{w\in L}w.
\end{eqnarray}
In particular, one still denotes, for convenience, the \textit{characteristic series} of $\calX$ (resp. $\calX^*$) by $\calX$ (resp. $\calX^*$) and then
\begin{eqnarray}\label{X+}
\calX^+=\calX^*-1_{\calX^*}.
\end{eqnarray}

The set of series over $\calX$ with coefficients in $A$ is denoted by \cite{berstel}
\begin{eqnarray}
\ncs{A}{\calX}=A^{\calX^*}.
\end{eqnarray}
Then, for any $\lambda\in A$, $S$ and $R\in\ncs{A}{\calX}$, the mutiplication $S$ by $\lambda$ and the sum of $S$ and $R$ (both belonging $\ncs{A}{\calX}$) are given, for any $w\in\calX^*$, respectively by
\begin{eqnarray}
\scal{\lambda S}{w}=\lambda\scal{S}{w}&\mbox{and}&\scal{S+T}{w}=\scal{S}{w}+\scal{T}{w}.
\end{eqnarray}
The set of polynomials over $\calX$ with coefficients in $A$ is an $A$-module, attmiting $\{w\}_{w\in X^*}$ as linear basis, and is denoted by \cite{berstel}
\begin{eqnarray}
\ncp{A}{\calX}\cong A[\calX^*].
\end{eqnarray}

In the sequel, all algebras, linear maps and tensor signs that appear in the following are over $A$ unless specified otherwise.

By the following \textit{pairing}, in which the sum is finite because $P$ is a polynomial,
\begin{eqnarray}
\ncs{A}{\calX}\otimes\ncp{A}{\calX}\longrightarrow A,&&
T\otimes P\longmapsto\scal{T}{P}:=\sum_{w\in\calX^*}\scal{T}{w}\scal{P}{w},\label{pairing}
\end{eqnarray}
there is a natural duality between $\ncp{A}{\calX}$ and $\ncs{A}{\calX}$, \textit{i.e.} \cite{berstel}
\begin{eqnarray}
\ncs{A}{\calX}=\ncp{A}{\calX}^{\vee}.
\end{eqnarray}
From \eqref{pairing}, using the Kronecker delta, it follows that
\begin{eqnarray}
\forall u,v\in\calX^+,&\scal{u}{v}=\delta_{u,v}.
\end{eqnarray}

Let $\ncs{A}{\calX}$ be equipped the ultrametric distance defined by \cite{berstel}
\begin{eqnarray}\label{distance}
\forall S,T\in\ncs{A}{\calX},&d(S,T)=2^{-\omega(S-T)}.
\end{eqnarray}
where $\omega(S)$ denotes the valuation of $S\in\ncs{A}{\calX}$ defined by \cite{berstel}
\begin{eqnarray}
\omega(S)&:=&\left\{\begin{array}{lcl}
+\infty&\mbox{if}&S=0,\\
\inf\{\abs{w}\}_{w\in\supp S}&\mbox{if}&S\neq0,
\end{array}\right.
\end{eqnarray}
For the discrete topology defined in \eqref{distance}, $\ncs{A}{\calX}$ is a complete topological ring and $\ncp{A}{\calX}$ is a dense subring of $\ncs{A}{\calX}$, \textit{i.e} \cite{berstel}
\begin{eqnarray}
\widehat{\ncp{A}{\calX}}=\ncs{A}{\calX}.
\end{eqnarray}

Let $\ncp{\calL ie_A}{\calX}$ be the smallest $A$-submodule of $\ncp{A}{\calX}$, contaning $\calX$ and being closed by the Lie bracket defined, for any $P$ and $Q\in\ncp{A}{\calX}$, as follows \cite{lothaire,reutenauer,viennotgerard}
\begin{eqnarray}
[P,Q]=PQ-QP.
\end{eqnarray}
This bracket is anticommutative and satisfies the Jacobi identity \cite{lothaire,reutenauer,viennotgerard}. Any element $P$ of $\ncp{\calL ie_A}{\calX}$ is called \textit{Lie polynomial} and it is propre, \textit{i.e.} $\scal{P}{1_{\calX^*}}=0$. It is also proved that $\ncp{\calL ie_A}{\calX}$ is the free Lie algebra over $A$ \cite{lothaire,reutenauer,viennotgerard}.

A series $S\in\ncs{A}{\calX}$ is a \textit{Lie series} if it is uniquely expressed as follows \cite{reutenauer}
\begin{eqnarray}
S=\sum_{k\ge1}P_k,
\end{eqnarray}
where $\{P_k\}_{k\ge1}$ are Lie polynomials homogenous of weight $k$ \cite{lothaire,reutenauer,viennotgerard}. The Lie algebra of Lie series is denoted by $\ncs{\calL ie_A}{\calX}$ and is endowed the bracket defined by
\begin{eqnarray}
S=\sum_{k\ge1}P_k\mbox{ and }R=\sum_{l\ge1}Q_l,&&
[S,R]=\sum_{k,l\ge1}[P_k,Q_l],
\end{eqnarray}

As algebras, $\ncp{A}{\calX}$ is also equipped
\begin{itemize}
\item The associative noncommutative and unital concatenation, \textit{i.e.} the following bilinear map
\begin{eqnarray}\label{conc}
\conc:\ncp{A}{\calX}\otimes\ncp{A}{\calX}&\longrightarrow&\ncp{A}{\calX}
\end{eqnarray}
or, equivalently, by the coproduct (with respected to the pairing in \eqref{pairing})
\begin{eqnarray}\label{Dconc}
\Delta_{\conc}:\ncp{A}{\calX}\longrightarrow\ncp{A}{\calX}\otimes\ncp{A}{\calX},&&
x\longmapsto1_{\calX^*}\otimes x+x\otimes1_{\calX^*}
\end{eqnarray}
such that, for any $u,v,w\in\calX^*$ as follows
\begin{eqnarray}
\scal{{\tt conc}(u,v)}{w}=\scal{uv}{\Delta_{\conc}w}.
\end{eqnarray}
$\Delta_{\conc}$ is a morphism for concatenation and then, for any $w\in\calX^*$, one has
\begin{eqnarray}
\Delta_{\conc}w=\sum_{u,v\in\calX^*,uv=w}u\otimes v.
\end{eqnarray}

\item The associative commutative and unital shuffle product, \textit{i.e.} 
\begin{eqnarray}\label{shuffle}
\shuffle:\ncp{A}{\calX}\otimes\ncp{A}{\calX}\longrightarrow\ncp{A}{\calX}
\end{eqnarray}
defined, for any $x,y\in\calX$ and $u,v,w\in\calX^*$, by the following recursion
\begin{eqnarray}
&w\shuffle 1_{\calX^*}=1_{\calX^*}\shuffle w=w&\\
&xu\shuffle yv=x(u\shuffle yv)+y(xu\shuffle v),&
\end{eqnarray}
or equivalently, the coproduct (with respected to the pairing in \eqref{pairing})
\begin{eqnarray}\label{Dshuffle}
\Delta_{\shuffle}:\ncp{A}{\calX}\longrightarrow\ncp{A}{\calX}\otimes\ncp{A}{\calX},&&
x\longmapsto{\calX^*}\otimes x+x\otimes1_{\calX^*}
\end{eqnarray}
and such that, for any $u,v,w\in\calX^*$,
\begin{eqnarray}
\scal{u\shuffle v}{w}=\scal{uv}{\Delta_{\shuffle}(w)}.
\end{eqnarray}
It is a morphism for concatenation, \textit{i.e.}, for any $u,v\in\calX^*$,
\begin{eqnarray}
\Delta_{\shuffle}(uv)=(\Delta_{\shuffle}u)(\Delta_{\shuffle}v).
\end{eqnarray}
\end{itemize}

\begin{definition}[\cite{SLC74,siblings}]\label{phishuffle}
Let us consider the following structure constants being finite
\begin{eqnarray*}
\forall z\in Y,&\#\{(x,y)\in Y^2|\gamma_{x,y}^z\not=0\}<+\infty
\end{eqnarray*}
and the following associative commutative map
\begin{eqnarray*}
\phi:Y\times Y\longrightarrow A Y,&&
(y_i,y_j)\longmapsto\sum_{k\in I}\gamma_{i,j}^ky_k.
\end{eqnarray*}

The map ${\shuffle}_{\phi}:Y^*\times Y^*\longrightarrow\ncp{A}{Y}$ defines a product and is extended as a law satisfying, for any $u,v\in Y^*$ and $y_i,y_j\in Y$, the following recursion
\begin{eqnarray*}
&u{\shuffle}_{\phi}1_{Y^*}=1_{Y^*}{\shuffle}_{\phi}u=u,&\\
&y_iu{\shuffle}_{\phi}y_jv=y_i(u{\shuffle}_{\phi}y_jv)+y_j(y_iu{\shuffle}_{\phi} v)+\phi(y_i,y_j)(u{\shuffle}_{\phi} v).&
\end{eqnarray*}
\end{definition}

Additionally, let $\ncp{A}{Y}$ be also equipped with the associative commutative and unital $\phi$-deformations shuffle product, \textit{i.e.} the following bilinear map
\begin{eqnarray}\label{stuffle}
{\shuffle}_{\phi}:\ncp{A}{Y}\otimes\ncp{A}{Y}\longrightarrow\ncp{A}{Y},
\end{eqnarray}
satisfying two items of Definition \ref{phishuffle}. Or equivalently, the coproduct (with respected to the pairing in \eqref{pairing})
\begin{eqnarray}\label{Dstuffle}
\Delta_{{\shuffle}_{\phi}}:\ncp{A}{Y}&\longrightarrow&\ncp{A}{Y}\otimes\ncp{A}{Y},\\
y_k&\longmapsto&y_k\otimes 1_{Y^*}+1_{Y^*}\otimes y_k+\sum\limits_{i+j=k}\gamma_{i,j}^ky_i\otimes y_j
\end{eqnarray}
and such that, for any $u,v,w\in Y^*$,
\begin{eqnarray}
\scal{u{\shuffle}_{\phi} v}{w}=\scal{uv}{\Delta_{{\shuffle}_{\phi}}(w)}.
\end{eqnarray}
The coproduct $\Delta_{{\shuffle}_{\phi}}$ is also a $\conc$-morphism, \textit{i.e.}, for any $u,v\in Y^*$,
\begin{eqnarray}
\Delta_{{\shuffle}_{\phi}}uv=(\Delta_{{\shuffle}_{\phi}}u)(\Delta_{{\shuffle}_{\phi}}v).
\end{eqnarray}

\begin{proposition}[\cite{SLC74,siblings}]
One has
\begin{itemize}
\item ${\shuffle}_{\phi}$ is commutative if and only if
$\phi:A{Y}\times A{Y}\longrightarrow{A}{{Y}}$ is so.
\item ${\shuffle}_{\phi}$ is associative if and only if
$\phi:A{Y}\times A{Y}\longrightarrow{A}{{Y}}$ is so.
\item ${\shuffle}_{\phi}$ is dualizable if and only if $\{(\gamma_{x,y}^z\}_{x,y,z\in {Y}}$ satisfy
\begin{eqnarray*}
(\forall z\in {Y})(\#\{(x,y)\in{Y}^2|\gamma_{x,y}^z\not=0\}<+\infty).
\end{eqnarray*}
\end{itemize}
\end{proposition}

\begin{example}[$q$-deformation of shuffle product, \cite{SLC74}]\label{q}
Let $q\in\C$ and $\phi$ be defined by
\begin{eqnarray*}
\forall y_i,y_j\in Y,&\phi(y_i,y_j)=q^{i+j}.
\end{eqnarray*}
Then one obtains the $q$-deformation of shuffle product defined, for any $y_i,y_j\in Y^*$ and $u,v\in Y^*$, by
\begin{eqnarray*}
&u{\shuffle}_q1_{Y^*}=1_{Y^*}{\shuffle}_qu=u,\\
&y_iu{\shuffle}_qy_jv=y_i(u{\shuffle}_q)v+y_j(y_iu{\shuffle}_qv)+q^{i+j}u{\shuffle}_qv.
\end{eqnarray*}

In particular, ${\shuffle}_q$ corresponds to 
\begin{itemize}
\item the ordinary shuffle, denoted by $\shuffle$, for $q=0$,
\item the quasi-shuffle, denoted by $\stuffle$, for $q=1$,
\item the minus-stuffle product, denoted by $\ministuffle$, for $q=-1$ .
\end{itemize}
\end{example}

In all the sequel, let $\phi$ be supposed associative, commutative and dualizable.

Now, let us extend the above products (\textit{i.e.} $\conc,\shuffle,{\shuffle}_{\phi}$ in \eqref{conc}, \eqref{shuffle}, \eqref{stuffle}, respectively)
\begin{eqnarray}
\conc,\shuffle:\ncs{A}{\calX}\otimes\ncs{A}{\calX}&\longrightarrow&\ncs{A}{\calX},\\
{\shuffle}_{\phi}:\ncs{A}{Y}\otimes\ncs{A}{Y}&\longrightarrow&\ncs{A}{Y}
\end{eqnarray}
as follows
\begin{eqnarray}
\forall S,R\in\ncs{A}{\calX},&SR&=\sum_{w\in\calX^*}\biggl(\sum_{u,v\in\calX^,uv=w}\scal{S}{u}\scal{R}{v}\biggr)w,\label{extenoverseries1}\\
\forall S,R\in\ncs{A}{\calX},&S\shuffle R&=\sum_{u,v\in\calX^*}\scal{S}{u}\scal{R}{v}u\shuffle v,\label{extenoverseries2}\\
\forall S,R\in\ncs{A}{Y},&S{\shuffle}_{\phi} R&=\sum_{u,v\in Y^*}\scal{S}{u}\scal{R}{v}u{\shuffle}_{\phi} v\label{extenoverseries3}.
\end{eqnarray}

For any propre series $S$, using the fact that $S^0=S^{\shuffle 0}=1_{\calX^*}$ and, for any $k>1$,
\begin{itemize}
\item $S^k=S^{k-1}S=SS^{k-1}$ and $S^{\shuffle k}=S^{\shuffle(k-1)}S=SS^{\shuffle(k-1)}$,
\item $S^{\shuffle k}=k!S^k$,
\end{itemize}
one also defines (see also Definition \ref{diaser} below)
\begin{eqnarray}\label{explog}
\exp S=\sum_{k\ge0}\frac{S^k}{k!}&\mbox{and}&
\exp_{\shuffle}S=\sum_{k\ge0}\frac{S^{\shuffle k}}{k!}=\sum_{k\ge0}S^k,\\
\log(1_{\calX^*}+S)=\sum_{k\ge1}\frac{(-1)^{k-1}}{k}S^k
&\mbox{and}&
\log_{\shuffle}(1_{\calX^*}+S)=\sum_{k\ge0}{(-1)^k}{k!}S^{k+1}.
\end{eqnarray}

The coproducts $\Delta_{\conc},\Delta_{\shuffle},\Delta_{{\shuffle}_{\phi}}$ respectively in \eqref{Dconc}, \eqref{Dshuffle}, \eqref{Dstuffle} are extended by
\begin{eqnarray}
\Delta_{\conc},\Delta_{\shuffle}:\ncs{A}{\calX}&\longrightarrow&\ncs{A}{\calX^*\otimes\calX^*},\\
\Delta_{{\shuffle}_{\phi}}:\ncs{A}{Y}&\longrightarrow&\ncs{A}{Y\otimes Y}
\end{eqnarray}
such that
\begin{eqnarray}
\forall S\in\ncs{A}{\calX},&\Delta_{\conc}S&=\sum_{w\in\calX^*}\scal{S}{w}\Delta_{\conc}w\in{\ncs{A}{\calX^*\otimes\calX^*}},\label{D3}\\
\forall S\in\ncs{A}{\calX},&\Delta_{\shuffle}S&=\sum_{w\in\calX^*}\scal{S}{w}\Delta_{\shuffle}w\in{\ncs{A}{\calX^*\otimes\calX^*}},\label{D2}\\
\forall S\in\ncs{A}{Y},&\Delta_{{\shuffle}_{\phi}}S&=\sum_{w\in Y^*}\scal{S}{w}\Delta_{{\shuffle}_{\phi}}w\in{\ncs{A}{Y^*\otimes Y^*}}.\label{D1}
\end{eqnarray}

\begin{remark}
Let $A=K$ be a field.

Then $\ncs{K}{\calX}\otimes\ncs{K}{\calX}$ embeds (injectively) in $\ncs{K}{\calX^*\otimes\calX^*}\cong\ncs{[\ncs{K}{\calX}]}{\calX}$.

Indeed, $\ncs{K}{\calX}\otimes\ncs{K}{\calX}$ contains the elements of the form $\sum\limits_{i\in I}G_i\otimes D_i$, for some finite $I$ and $(G_i,D_i)\in\ncs{K}{\calX}\times\ncs{K}{\calX}$. But, for any $u$ and $v\in\calX^+$, one has
\begin{eqnarray*}
S=\sum\limits_{i\ge0}u^i\otimes v^i\in\ncs{K}{\calX^*\otimes\calX^*}
&\mbox{and}&S\notin\ncs{K}{\calX}\otimes\ncs{K}{\calX}.
\end{eqnarray*}
\end{remark}

\begin{definition}\label{gpl_p}
A series $S\in\ncs{A}{Y}$ (resp. $\ncs{A}{\calX}$) is said to be 
\begin{enumerate}
\item grouplike for $\Delta_{{\shuffle}_{\phi}}$ (resp. $\Delta_{\shuffle}$ and $\Delta_{\conc}$), if and only if
\begin{enumerate}
\item $\scal{S}{1_{Y^*}}=1_{A}$ (resp. $\scal{S}{1_{c^*}}=1_{A}$),
\item $\Delta_{{\shuffle}_{\phi}}S=S\otimes S$ (resp. $\Delta_{\shuffle}S=S\otimes S$ and $\Delta_{\conc}S=S\otimes S$).
\end{enumerate}

\item primitive for $\Delta_{{\shuffle}_{\phi}}$ (resp. $\Delta_{\shuffle}$ and $\Delta_{\conc}$), if and only if
\begin{eqnarray*}
\Delta_{{\shuffle}_{\phi}}S&=&1_{Y^*}\otimes S+S\otimes1_{Y^*}\\
\mbox{(resp. }\Delta_{\shuffle}S&=&1_{\calX^*}\otimes S+S\otimes1_{\calX^*}\\ \mbox{and }\Delta_{\conc}S&=&1_{\calX^*}\otimes S+S\otimes1_{\calX^*}).
\end{eqnarray*}

\item Let $\calG_{{\shuffle}_{\phi}}^Y$ (resp. $\calG_{\shuffle}^{\calX}$ and $\calG_{\conc}^{\calX}$) denote the set of grouplike series for $\Delta_{{\shuffle}_{\phi}}$ (resp. $\Delta_{\shuffle}$ and $\Delta_{\conc}$) and $\calP_{{\shuffle}_{\phi}}^Y$ (resp. $\calP_{\shuffle}^{\calX}$ and $\calP_{\conc}^{\calX}$) denote the set of primitive series for $\Delta_{{\shuffle}_{\phi}}$ (resp. $\Delta_{\shuffle}$ and $\Delta_{\conc}$).
\end{enumerate}
\end{definition}

\begin{remark}\label{R0}
By \eqref{Dconc} and \eqref{Dshuffle}, any letter $x\in\calX$ is primitive for $\Delta_{\conc}$ and $\Delta_{\shuffle}$. By \eqref{Dstuffle}, the letter $y_1$ is primitive for $\Delta_{{\shuffle}_{\phi}}$.
\end{remark}

\begin{proposition}\label{proposition1}
$\calG_{{\shuffle}_{\phi}}^Y$ (resp. $\calG_{\shuffle}^{\calX}$ and $\calG_{\conc}^{\calX}$) is a group and $\calP_{{\shuffle}_{\phi}}^Y$ (resp. $\calP_{\shuffle}^{\calX}$ and $\calP_{\conc}^{\calX}$) is a Lie algebra.
\end{proposition}

\begin{proof}
It is classical in theory of Hopf algebra \cite{abe,CartierPatras} (see also \cite{VJM,reutenauer}).
\end{proof}

\begin{definition}\label{dec0}
For ${\shuffle}_{\phi}$ (resp. $\shuffle$ and $\conc$), a series $S\in\ncs{A}{Y}$ (resp. $\ncs{A}{\calX}$) is
\begin{enumerate}
\item a character of $\ncp{A}{Y}$ (resp. $\ncp{A}{\calX}$) if and only if, for any $u$ and $v\in Y^*$ (resp. $\calX^*$), one has
\begin{enumerate}
\item $\scal{S}{1_{Y^*}}=1_A$ (resp. $\scal{S}{1_{\calX^*}}\allowbreak=1_A$),
\item $\scal{S}{u{\shuffle}_{\phi} v}=\scal{S}{u}\scal{S}{v}$
(resp. $\scal{S}{u\shuffle v}=\scal{S}{u}\scal{S}{v}$ and $\scal{S}{uv}=\scal{S}{u}\scal{S}{v}$).
\end{enumerate}

\item an infinitesimal character of $\ncp{A}{Y}$ (resp. $\ncp{A}{\calX}$) if and only if, for any $u$ and $v\in Y^*$ (resp. $\calX^*$), one has
\begin{eqnarray*}
\scal{S}{u{\shuffle}_{\phi} v}&=&\scal{S}{u}\scal{v}{1_{Y^*}}+\scal{u}{1_{Y^*}}\scal{S}{v},\\
\mbox{(resp. }\scal{S}{u\shuffle v}&=&\scal{S}{u}\scal{v}{1_{Y^*}}+\scal{u}{1_{Y^*}}\scal{S}{v},\\
\mbox{and }\scal{S}{uv}&=&\scal{S}{u}\scal{v}{1_{Y^*}}+\scal{u}{1_{Y^*}}\scal{S}{v}).
\end{eqnarray*}
\end{enumerate}
\end{definition}

\begin{proposition}\label{proposition2}
Let $S\in\ncs{A}{\calX}$. 
\begin{enumerate}
\item $S$ is grouplike series for $\Delta_{{\shuffle}_{\phi}}$ (resp. $\Delta_{\shuffle}$ and $\Delta_{\conc}$) if and only if $S$ is a character of $\ncp{A}{Y}$ (resp. $\ncp{A}{\calX}$) for ${\shuffle}_{\phi}$ (resp. $\shuffle$ and $\conc$).

\item $S$ is primitive series for $\Delta_{{\shuffle}_{\phi}}$ (resp. $\Delta_{\shuffle}$ and $\Delta_{\conc}$) if and only if $S$ is an infinitesimal character of $\ncp{A}{Y}$ (resp. $\ncp{A}{\calX}$) for $\shuffle_{\phi}$ (resp. $\shuffle$ and $\conc$). 

\item\label{ree} If $\scal{S}{1_{\calX^*}}=1$ then $S$ is grouplike if and only if $\log S$ is primitive for $\Delta_{{\shuffle}_{\phi}}$ (resp. $\Delta_{\shuffle}$).
\end{enumerate}
\end{proposition}

\begin{proof}
It is classical in the theory of Hopf algebra \cite{abe,CartierPatras} (see also \cite{VJM,reutenauer}).
\end{proof}

\begin{remark}\label{Ree}
For $\Delta_{\shuffle}$, Item \ref{ree} of Proposition \ref{proposition2} is known as a theorem by Ree \cite{reutenauer} and it is extended for $\Delta_{\stuffle}$ in \cite{VJM}. For $\Delta_{\conc}$, one can see Corollary \ref{Kleene} below.
\end{remark}

\begin{definition}\label{diaser}
Let $S\in\ncs{A}{\calX}$ (resp. $\ncs{A}{\calX}\otimes\ncs{A}{\calX}$).
\begin{enumerate}
\item If $\scal{S}{1_{\calX^*}}=0$ (resp. $\scal{S}{1_{\calX^*}\otimes1_{\calX^*}}=0$) then one defines the Kleene star of $S$ by the following infinite sum
\begin{eqnarray*}
S^*:=1+S+S^2+\cdots.
\end{eqnarray*}

\item In the same way, one also defines the diagonal series by
\begin{eqnarray*}
\calD_{\calX}:=\calM_{\calX}^*,&\mbox{where }&
\calM_{\calX}:=\sum_{t\in\calX}t\otimes t\mbox{ and then }
\calM_{\calX}^*=\sum_{w\in\calX^*}w\otimes w.
\end{eqnarray*}

\item As in \eqref{X+}, one also defines
\begin{eqnarray*}
\calM_{\calX}^+:=\calM_{\calX}^*-1_{\calX^*}\otimes1_{\calX^*}.
\end{eqnarray*}
\end{enumerate}
\end{definition}

\begin{remark}
If $S\in\ncs{A}{\calX}$ such that $\scal{S}{1_{\calX^*}}=0$ then, by \eqref{explog}, $S^*=\exp_{\shuffle}S$.
\end{remark}

For any $S\in\ncs{A}{\calX}$ (resp. $\ncs{A}{\calX}\otimes\ncs{A}{\calX}$) such that $\scal{S}{1_{\calX^*}}=0$ (resp. $\scal{S}{1_{\calX^*}\otimes1_{\calX^*}}=0$), setting $\nabla T:=T-1_{\calX^*}$ (resp. $\nabla T:=T-1_{\calX^*}\otimes1_{\calX^*}$), the Kleene star $S^*$ is the unique\footnote{Solutions obtained by convergent iteration process for a discrete topology as in \eqref{distance}.} solution of the following left and right star equations
\begin{eqnarray}
\nabla T=\calX T&\mbox{and}&\nabla T=T\calX
\end{eqnarray}
and, similarly, the diagonal series $\calD_{\calX}$ is the unique solution of the following left and right star equations ($T$ is unknown series)
\begin{eqnarray}
\nabla T=\calM_{\calX}T&\mbox{and}&\nabla T=T\calM_{\calX}.
\end{eqnarray}

In order to prove Theorem \ref{isomorphy} below, the following product in $\ncp{A}{\calX}\otimes\ncp{A}{\calX}$ (resp. $\ncp{A}{Y}\otimes\ncp{A}{Y}$) will  be also considered, for any $u_i,v_i\in\calX^+$ (resp. $Y^+$), $i=1,2$,
\begin{eqnarray}
(u_1\otimes v_1)(u_2\otimes v_2)&=&(u_1\shuffle u_2)\otimes(u_2v_2),\label{tensor1}\\
(\mbox{resp. }
(u_1\otimes v_1)(u_2\otimes v_2)&=&(u_1{\shuffle}_{\phi}u_2)\otimes(u_2v_2)).\label{tensor2}
\end{eqnarray}

\begin{example}\label{tensor}
For any $k\ge1$, one obtains
\begin{eqnarray*}
\Big(\sum_{w\in\calX^+}w\otimes w\Big)^k&=&\sum_{u_1,\ldots,u_k\in\calX^+}
(u_1{\shuffle}\ldots{\shuffle}u_k)\otimes u_1\ldots u_k\\
\Big(\mbox{resp. }
\Big(\sum_{w\in Y^+}w\otimes w\Big)^k&=&\sum_{u_1,\ldots,u_k\in Y^+}
(u_1{\shuffle}_{\phi}\ldots{\shuffle}_{\phi} u_k)\otimes u_1\ldots u_k\Big).\end{eqnarray*}
\end{example}

\subsection{Various bialgebras}\label{bialgebras}
In the theory of Hopf algebras, the question of primitive elements is an important problem (see Defintions \ref{gpl_p} and \ref{dec0}, Propositions \ref{proposition1} and \ref{proposition2}) and the following CQMM\footnote{CQMM is an abbreviation of P. Cartier, D. Quillen, J. Milnor and J. Moore.} theorem \cite{Lie7,Cartier2,CartierPatras,MilnorMoore} provides necessary and sufficient conditions for a bialgebra to be an enveloping algebra \cite{Lie7,dixmier}.

\begin{theorem}[CQMM theorem, \cite{SLC74}]\label{CQMM}
Let $A$ be an unitary commutative associative $\Q$-algebra and $\calB$ be a noncommutative co-commutative $A$-bialgebra\footnote{Applcations below concern $(\ncp{A}{\calX},\conc,1_{\calX^*},\Delta_{\shuffle})$ and $(\ncp{A}{Y},\conc,1_{Y^*},\Delta_{{\shuffle}_{\phi}})$.}. Let $\calU(\calP_{\times})$ be the enveloping algebra generated by the primitive elements of $\calB$ for $\Delta_{\times}$ (\textit{i.e.} $\calP_{\times}$).

Then the following assertions are equivalent.
\begin{enumerate}
\item $\mathcal{B}$ is isomorphic to the enveloping algebra $\calU(\calP_{\times})$.

\item There is an increasing sequence $\{\calB_n\}$,
$\calB_0=A.1_{\calB}\subset\calB_1\subset\cdots\subset\calB_n\subset\cdots$,
satisfying
\begin{enumerate}
\item $\calB=\bigcup\limits_{n\ge0}\calB_n$,
\item $\forall p,q\in\N,\calB_p\calB_q\subset\calB_{p+q}$,
\item $\forall n\in\N,\Delta_{\times}(\calB_n)\subset\sum\limits_{p+q=n}\calB_p\otimes\calB_q$.
\end{enumerate}
\end{enumerate}
\end{theorem}

Let us consider the following bialgebras and their duals
\begin{eqnarray}\label{Bialgebras}
\calH_{\shuffle}(\calX)&:=&(\ncp{A}{\calX},\conc,1_{\calX^*},\Delta_{\shuffle}),\label{Bialgebras1}\\
\calH_{{\shuffle}_{\phi}}(Y)&:=&(\ncp{A}{Y},\conc,1_{Y^*},\Delta_{{\shuffle}_{\phi}})\label{Bialgebras2},\\
\calH_{\shuffle}^{\vee}(\calX)&:=&(\ncp{A}{\calX},{\shuffle},1_{\calX^*},\Delta_{\conc}),\label{dualbialgebras1}\\
\calH_{{\shuffle}_{\phi}}^{\vee}(Y)&:=&(\ncp{A}{Y},{{\shuffle}_{\phi}},1_{Y^*},\Delta_{\conc})\label{dualbialgebras2}.
\end{eqnarray}
By Propostion \ref{proposition1} and Theorem \ref{CQMM}, the enveloping algebra $\mathcal{U}(\calP_{\shuffle}^{\calX})$ (resp. $\mathcal{U}(\calP_{{\shuffle}_{\phi}}^Y)$) is isomorphic to the $A$-module associated to the bialgebras $\calH_{\shuffle}(\calX)$ (resp. $\calH_{{\shuffle}_{\phi}}(Y)$).

\begin{theorem}[Eulerian idempotent]\label{isomorphy}
Let $A$ be a $\Q$-algebra and $\pi_1:\ncp{A}{Y}\longrightarrow\ncp{A}{Y}$ be  the linear map such that $\pi_1(1_{Y^*})=1_{Y^*}$ and the image of $w\in Y^+$ is given by
\begin{eqnarray*}
\pi_1w=w+\sum_{k\ge2}\frac{(-1)^{k-1}}{k}
\sum_{u_1,\ldots,u_k\in Y^+}
\scal{w}{u_1{\shuffle}_{\phi}\ldots{\shuffle}_{\phi}u_k}u_1\ldots u_k.
\end{eqnarray*}
Then one has
\begin{enumerate}
\item $\Delta_{{\shuffle}_{\phi}}\pi_1w=1_{Y^*}\otimes\pi_1w+\pi_1w\otimes1_{Y^*}$ and $\log\calD_Y=(\mathrm{Id}\otimes\pi_1)\calD_Y$.

\item For any $w\in Y^+$,
\begin{eqnarray*}
w=\sum_{k\ge0}\frac1{k!}\sum_{u_1,\ldots,u_k\in Y^+}\scal{w}{u_1{\shuffle}_{\phi}\cdots{\shuffle}_{\phi}u_k}(\pi_1u_1)\cdots(\pi_1u_k).
\end{eqnarray*}

\item The endomorphism $\varphi_{\pi_1}:(\ncp{A}{Y},\conc,1_{Y^*})\longrightarrow(\ncp{A}{Y},\conc,1_{Y^*})$, mapping $y_k$ to $\pi_1y_k$, is an isomorphism of bialgebras between $\calH_{\shuffle}(Y)$ and $\calH_{{\shuffle}_{\phi}}(Y)$.

\item $\Delta_{{\shuffle}_{\phi}}\circ\varphi_{\pi_1}=(\varphi_{\pi_1}\otimes\varphi_{\pi_1})\circ\Delta_{{\shuffle}_{\phi}}$.
\end{enumerate}
\end{theorem}

\begin{proof}
\begin{enumerate}
\item With the notations in Definition \ref{diaser}, as in \cite{VJM}, since $\calD_Y=1+\calM_Y^+$ then expanding $\log\calD_Y=\log(1+\calM_Y^+)$ in $\calH_{{\shuffle}_{\phi}}(Y)$, one has
\begin{eqnarray*}
\log\calD_Y=\sum_{k\ge1}\frac{(-1)^{k-1}}{k}\Big(\sum_{w\in Y^+}w\otimes w\Big)^k.
\end{eqnarray*}
Using \eqref{tensor2} and the calulations in Example \ref{tensor}, one obtains
\begin{eqnarray*}
\log\calD_Y
=\sum_{w\in Y^+}w\otimes\underbrace{\sum_{k\ge1}\frac{(-1)^{k-1}}{k}
\sum_{u_1,\ldots,u_k\in Y^+}
\scal{w}{u_1{\shuffle}_{\phi}\ldots{\shuffle}_{\phi} u_k}u_1\ldots u_k}_{=\pi_1w}.
\end{eqnarray*}
It follows that $\log\calD_Y=(\mathrm{Id}\otimes\pi_1)\calD_Y$. 

Now, let $S$ be a grouplike series for $\Delta_{{\shuffle}_{\phi}}$ such that, viewed as tensor products of ismorphisms of algebras, $S\otimes\mathrm{Id}$ and $S\otimes\pi_1$ act on the diagonal series $\calD_Y$ as follows
\begin{eqnarray*}
S=&(S\otimes\mathrm{Id})\calD_Y&=\sum_{w\in Y^*}\scal{S}{w}w,\cr
\log S=&(S\otimes\pi_1)\calD_Y&=\sum_{w\in Y^*}\scal{S}{w}\pi_1w. 
\end{eqnarray*}
Since $\log S$ is primitive for $\Delta_{{\shuffle}_{\phi}}$ (see Proposition \ref{proposition2}) then
\begin{eqnarray*}
\Delta_{{\shuffle}_{\phi}}(\log S)=1_{Y^*}\otimes(\log S)+(\log S)\otimes1_{Y^*}.
\end{eqnarray*}
On the right side, expanding $\log S$ as above, one has
\begin{eqnarray*}
1_{Y^*}\otimes(\log S)=1_{Y^*}\otimes\Big(\sum_{w\in Y^*}\scal{S}{w}\pi_1w\Big)=\sum_{w\in Y^*}\scal{S}{w}(1_{Y^*}\otimes\pi_1w),\cr
(\log S)\otimes1_{Y^*}=\Big(\sum_{w\in Y^*}\scal{S}{w}\pi_1w\Big)\otimes1_{Y^*}=\sum_{w\in Y^*}\scal{S}{w}(\pi_1w\otimes1_{Y^*}).
\end{eqnarray*}
On the left side, by \eqref{D1}, one also has
\begin{eqnarray*}
\Delta_{{\shuffle}_{\phi}}(\log S)=\sum_{w\in Y^*}\scal{S}{w}\Delta_{{\shuffle}_{\phi}}(\pi_1w).
\end{eqnarray*}
Thus, identifying term by term, one obtains finally
\begin{eqnarray*}
\Delta_{{\shuffle}_{\phi}}(\pi_1w)=1_{Y^*}\otimes\pi_1w+\pi_1w\otimes1_{Y^*}.
\end{eqnarray*}

\item Since $\calD_Y=\exp(\log\calD_Y)$ then, using \eqref{tensor2} and the previous item, one deduces the expected result as follows
\begin{eqnarray*}
\calD_Y&=&\sum_{k\ge0}\frac1{k!}\Big(\sum_{w\in Y^+}w\otimes\pi_1w\Big)^k\cr
&=&\sum_{k\ge0}\frac1{k!}\sum_{u_1,\cdots,u_k\in Y^+}
(u_1{\shuffle}_{\phi}\cdots{\shuffle}_{\phi}u_k)\otimes(\pi_1u_1)\cdots(\pi_1u_k)\cr
&=&\sum_{w\in Y^+}w\otimes\sum_{k\ge0}\frac1{k!}\sum_{u_1,\cdots,u_k\in Y^+}\scal{w}{u_1{\shuffle}_{\phi}\cdots{\shuffle}_{\phi}u_k}(\pi_1u_1)\cdots(\pi_1u_k).
\end{eqnarray*}

\item By the previous item, for each letter $y_s\in Y$, one obtains
\begin{eqnarray*}
y_s
&=&\sum_{k\ge1}\frac{1}{k!}\sum_{s_1+\cdots+s_k=s}(\pi_1y_1)\cdots(\pi_1y_k)\cr
&=&\sum_{k\ge1}\frac{1}{k!}\sum_{s_1+\cdots+s_k=s}\varphi_{\pi_1}(y_1)\cdots\varphi_{\pi_1}(y_k)\cr
&=&\sum_{k\ge1}\frac{1}{k!}\sum_{s_1+\cdots+s_k=s}\varphi_{\pi_1}(y_1\cdots y_k)
\end{eqnarray*}
meanning that the endomorphism $\varphi_{\pi_1}$ is surjective.
By definition of the linear map $\pi_1$, one also has
\begin{eqnarray*}
\varphi_{\pi_1}(y_s)=\pi_1y_s=y_s+\sum_{k\ge2}\frac{(-1)^{k-1}}{k}\sum_{j_1+\cdots+j_k=s}y_{j_1}\cdots y_{j_k}.
\end{eqnarray*}
Then, for any $w=y_{s_1}\cdots y_{s_r}\in Y^+$, one obtains
\begin{eqnarray*}
\varphi_{\pi_1}(w)=y_{s_1}\cdots y_{s_r}+
\sum_{k_1,\cdots,k_r\ge2}\frac{(-1)^{k_1+\cdots+k_r-r}}{k_1\cdots k_r}
\sum_{{j_{s_1,k_1}+\cdots+j_{s_1,k_r}=s_1}
\atop{\cdots\atop{j_{s_r,k_1}+\cdots+j_{s_r,k_r}=s_r}}}\\
(y_{j_{s_1,k_1}}\cdots y_{j_{s_1,k_1}})\cdots(y_{j_{s_r,k_r}}\cdots y_{j_{s_r,k_r}})
\end{eqnarray*}
meaning also that the $\varphi_{\pi_1}(w)$ is lower triangular:
\begin{eqnarray*}
\varphi_{\pi_1}(w)=w+\sum_{u<w,(u)=(w)}c_wu,&\mbox{for}&c_u\in A,
\end{eqnarray*}
and then $\varphi_{\pi_1}$ is injective.
Finally, $\varphi_{\pi_1}$ is an isomorphism of bialgebras.

\item The following diagram commutes
\begin{center}
\begin{tikzcd}[column sep=10em]
\ncp{A}{Y}\ar[hook]{r}{\Delta_{\shuffle}}\ar[swap]{d}{\varphi_{\pi_1}}& 
\ncp{A}{Y}\otimes\ncp{A}{Y}\ar[]{d}{\varphi_{\pi_1}\otimes\varphi_{\pi_1}}\\
\ncp{A}{Y}
\ar[hook]{r}{\Delta_{{\shuffle}_{\phi}}}& 
\ncp{A}{Y}\otimes\ncp{A}{Y}.
\end{tikzcd}
\end{center}
and as each member is an a morphism from $\ncp{A}{Y}$ to $\ncp{A}{Y}\otimes\ncp{A}{Y}$ then it suffices to prove the result on the letter $y_l$
\begin{eqnarray*}
\Delta_{{\shuffle}_{\phi}}(\varphi_{\pi_1}(y_l))
&=&\Delta_{{\shuffle}_{\phi}}\pi_1y_l\\
&=&(\pi_1y_l\otimes)1_{Y^*}+1_{Y^*}\otimes(\pi_1y_l)\\
&=&\varphi_{\pi_1}(y_l)\otimes1_{Y^*}+1_{Y^*}\otimes\varphi_{\pi_1}(y_l)\\
&=&(\varphi_{\pi_1}\otimes\varphi_{\pi_1})(y_l\otimes1_{Y^*}+1_{Y^*}\otimes y_l)\\
&=&(\varphi_{\pi_1}\otimes\varphi_{\pi_1})(\Delta_{{\shuffle}_{\phi}}y_l).
\end{eqnarray*}
\end{enumerate}
\end{proof}

Now, applying Theorems \ref{CQMM}--\ref{isomorphy}, let $\calX$ be equipped the following total orders
\begin{eqnarray}\label{orders}
x_0\prec\cdots\prec x_m&\mbox{and}&y_1\succ\cdots y_k\succ y_{k+1}\succ\cdots,
\end{eqnarray}
for which any $w\in\calX^+$ is a Lyndon word if it is is strictly smaller in lexicographic order than all of its rotations \cite{lothaire,viennotgerard}. Or equivalently, $w$ is a Lyndon word if and only if it is lexicographically strictly smaller than any of its propre suffixes that is, for any $u,v\in\calX^+$ such that $w=uv$, one has $w\prec v$ \cite{lothaire,viennotgerard}. The set of Lyndon words over $\calX$ is denoted by $\Lyn\calX$.
Any pair of Lyndon words $(l_1,l_2)$ is called the standard factorization of $l\in\Lyn\calX$, and is denoted by $st(l)$, if $l=l_1l_2$ and $l_2$ is the longest nontrivial propre right factor of $l$ or, equivalently, its smallest such (for the lexicographic ordering) \cite{lothaire}. According to a Radford's theorem (see \cite{reutenauer}), $\Lyn\calX$ forms a pure transcendence basis\footnote{\textit{I.e.} the $\{l\}_{l\in\Lyn\calX}$ are transcendent over $A$ and they are $A$-algebraically independent.} of the $A$-shuffle algebra $(\ncp{A}{\calX},\shuffle,1_{\calX^*})$.

If $A$ is a $\Q$-algebra then one classically endows $\ncp{A}{\calX}$ with the linear basis $\{P_w\}_{w\in \calX^*}$, expanded by the PBW theorem \cite{Lie7} after any basis $\{P_l\}_{l\in \Lyn\calX}$ of $\ncp{\calL ie_{A}}{\calX}$ and its dual basis $\{S_w\}_{w\in\calX^*}$ (containing the pure transcendence basis $\{S_l\}_{l\in\Lyn\calX}$ of the $A$-shuffle algebra) \cite{dixmier,reutenauer}. According to the lexicographical order on $\calX^*$, these dual bases $\{P_w\}_{w\in\calX^*}$ and $\{S_w\}_{w\in\calX^*}$ are homogeneous in weight\footnote{For ${\calX}=X$ or $=Y$ the corresponding monoids are equipped with length functions, for $X$ we consider the length of words and for $Y$ the length is given by the weight $\ell(y_{i_1}\ldots y_{i_n})=i_1+\ldots+i_n$. This naturally induces a grading of $\ncp{A}{{\calX}}$ and $\ncp{\calL ie_{A}}{{\calX}}$ in free modules of finite dimensions. For general ${\calX}$, we consider the fine grading \cite{reutenauer} \textit{i.e.} the grading by all partial degrees which, as well, induces a grading of $\ncp{A}{{\calX}}$ and $\ncp{\calL ie_{A}}{{\calX}}$ in free modules of finite dimensions.} \cite{reutenauer}.

These dual bases of polynomials $\{P_w\}_{w\in\calX^*}$ and $\{S_w\}_{w\in\calX^*}$, homogeneous in weight, can be constructed recursively as follows \cite{lothaire,viennotgerard}
\begin{eqnarray}\label{P_w}
\left\{\begin{array}{rclll}
P_x&=&x,
&\mbox{for }x\in\calX,\\
P_l&=&[P_{l_1},P_{l_2}],
&{\displaystyle\mbox{for }l=yl'\in\Lyn\calX\setminus\calX\atop\displaystyle st(l)=(l_1,l_2),}\\
P_w&=&P_{l_1}^{i_1}\ldots P_{l_k}^{i_k},
&{\displaystyle\mbox{for }w=l_1^{i_1}\ldots l_k^{i_k},\mbox{ with }l_1,\ldots,\atop\displaystyle l_k\in\Lyn\calX,l_1\succ\ldots\succ l_k.}
\end{array}\right.
\end{eqnarray}
and then by duality \cite{reutenauer},
\begin{eqnarray}\label{S_w}
\left\{\begin{array}{rclll}
S_x&=&x
&\mbox{for }x\in\calX,\\
S_l&=&yS_{l'},
&{\displaystyle\mbox{for }l=yl'\in\Lyn\calX\setminus\calX\atop\displaystyle st(l)=(l_1,l_2),}\\
S_w&=&\displaystyle\frac{S_{l_1}^{\shuffle i_1}\shuffle\ldots\shuffle S_{l_k}^{\shuffle i_k}}{i_1!\ldots i_k!},
&{\displaystyle\mbox{for }w=l_1^{i_1}\ldots l_k^{i_k},\mbox{ with }l_1,\ldots,\atop\displaystyle l_k\in\Lyn\calX,l_1\succ\ldots\succ l_k.}
\end{array}\right.
\end{eqnarray}

\begin{example}[\cite{hoangjacoboussous}]
For $X=\{x_0,x_1\}$, one has
$$\begin{array}{|c|c|c|}
\hline
l&P_l&S_l\\
\hline
x_0&x_0&x_0\\
x_1&x_1&x_1\\
x_0x_1&[x_0,x_1]&x_0x_1\\
x_0^2x_1&[x_0,[x_0,x_1]]&x_0^2x_1\\
x_0x_1^2&[[x_0,x_1],x_1]&x_0x_1^2\\
x_0^3x_1&[x_0,[x_0,[x_0,x_1]]]&x_0^3x_1\\
x_0^2x_1^2&[x_0,[[x_0,x_1],x_1]]&x_0^2x_1^2\\
x_0x_1^3&[[[x_0,x_1],x_1],x_1]&x_0x_1^3\\
x_0^4x_1&[x_0,[x_0,[x_0,[x_0,x_1]]]]&x_0^4x_1\\
x_0^3x_1^2&[x_0,[x_0,[[x_0,x_1],x_1]]]&x_0^3x_1^2\\
x_0^2x_1x_0x_1&[[x_0,[x_0,x_1]],[x_0,x_1]]&2x_0^3x_1^2+x_0^2x_1x_0x_1\\
x_0^2x_1^3&[x_0,[[[x_0,x_1],x_1],x_1]]&x_0^2x_1^3\\
x_0x_1x_0x_1^2&[[x_0,x_1],[[x_0,x_1],x_1]]&3x_0^2x_1^3+x_0x_1x_0x_1^2\\
x_0x_1^4&[[[[x_0,x_1],x_1],x_1],x_1]&x_0x_1^4\\
{x_0^5}x_1&[x_0,[x_0,[x_0,[x_0,[x_0,x_1]]]]]&x_0^5x_1\\\
x_0^4x_1^2&[x_0,[x_0,[x_0,[[x_0,x_1],x_1]]]]&x_0^4x_1^2\\
x_0^3x_1x_0x_1&[x_0,[[x_0,[x_0,x_1]],[x_0,x_1]]]&2x_0^4x_1^2+x_0^3x_1x_0x_1\\
x_0^3x_1^3&[x_0,[x_0,[[[x_0,x_1],x_1],x_1]]]&x_0^3x_1^3\\
x_0^2x_1x_0x_1^2&[x_0,[[x_0,x_1],[[x_0,x_1],x_1]]]&3x_0^3x_1^3+x_0^2x_1x_0x_1^2\\
x_0^2x_1^2x_0x_1&[[x_0,[[x_0,x_1],x_1]],[x_0,x_1]]&6x_0^3x_1^3+3x_0^2x_1x_0x_1^2+x_0^2x_1^2x_0x_1\\
x_0^2x_1^4&[x_0,[[[[x_0,x_1],x_1],x_1],x_1]]&x_0^2x_1^4\\
x_0x_1x_0x_1^3&[[x_0,x_1],[[[x_0,x_1],x_1],x_1]]&4x_0^2x_1^4+x_0x_1x_0x_1^3\\
x_0x_1^5&[[[[[x_0,x_1],x_1],x_1],x_1],x_1]&x_0x_1^5\\
\hline
\end{array}$$
\end{example}

$\Lyn Y$ forms a pure transcendence basis for $(\ncp{A}{Y},{\shuffle}_{\phi},1_{Y^*})$ \cite{siblings} and similar to the $\stuffle$-bialgebra \cite{VJM}, $\calP_{{\shuffle}_{\phi}}^Y$ is also endowed the linear basis $\{\Pi_w\}_{w\in Y^*}$, expanded by decreasing PBW basis after any basis $\{\Pi_l\}_{l\in \Lyn Y}$, homogeneous in weight, and its dual basis $\{\Sigma_w\}_{w\in Y^*}$ (containing the pure transcendence basis $\{\Sigma_l\}_{l\in\Lyn Y}$ of the $\phi$-deformations shuffle $A$-algebra) \cite{siblings}.

By Theorem \ref{isomorphy}, the bases of homogeneous in weight polynomials $\{\Pi_w\}_{w\in Y^*}$ and $\{\Sigma_w\}_{w\in Y^*}$ are images, respectively, of $\{P_w\}_{w\in Y^*}$ and $\{S_w\}_{w\in Y^*}$, by $\varphi_{\pi_1}$ and by the adjoint mapping of its inverse).

Algorithmically and similarly as in \cite{VJM,reutenauer}, these dual bases of polynomials can be constructed directly and recursively a follows
\begin{eqnarray}\label{Pi_w}
\left\{\begin{array}{rclll}
\Pi_{y_s}&=&\pi_1(y_s),
&\mbox{for }y_s\in Y,\\
\Pi_{l}&=&[\Pi_{l_1},\Pi_{l_2}],
&{\displaystyle\mbox{for }l\in\Lyn Y\setminus Y\atop\displaystyle st(l)=(l_1,l_2),}\\
\Pi_{w}&=&\Pi_{l_1}^{i_1}\ldots\Pi_{l_k}^{i_k},
&{\displaystyle\mbox{for }w=l_1^{i_1}\ldots l_k^{i_k},\mbox{ with }l_1,\ldots,\atop\displaystyle l_k\in\Lyn Y,l_1\succ\ldots\succ l_k}
\end{array}\right.
\end{eqnarray}
and then by duality,
\begin{eqnarray}\label{Sigma_w}
\left\{\begin{array}{rclll}
\Sigma_{y_s}&=&y_s
&\mbox{for }y_s\in Y,\\
\Sigma_l&=&\displaystyle\sum_{(**)}\frac{y_{s_{k_1}+\ldots+s_{k_i}}}{i!}\Sigma_{l_1\ldots l_n},
&{\displaystyle\mbox{for }l\in\Lyn Y\setminus Y\atop\displaystyle st(l)=(l_1,l_2),}\\
\Sigma_w&=&\displaystyle\frac{\Sigma_{l_1}^{{\shuffle}_{\phi}i_1}{\shuffle}_{\phi}\ldots{\shuffle}_{\phi}\Sigma_{l_k}^{{\shuffle}_{\phi}i_k}}{i_1!\ldots i_k!},
&{\displaystyle\mbox{for }w=l_1^{i_1}\ldots l_k^{i_k},\mbox{ with }l_1,\ldots,\atop\displaystyle l_k\in\Lyn Y,l_1\succ\ldots\succ l_k.}
\end{array}\right.
\end{eqnarray}
In $(**)$, the sum is taken over all $\{k_1,\ldots,k_i\}\subset\{1,\ldots,k\}$
and $l_1\succeq\ldots\succeq l_n$ such that $(y_{s_1},\ldots,y_{s_k})$ is derived from $(y_{s_{k_1}},\ldots,y_{s_{k_i}},l_1,\ldots,l_n)$ by  transitive closure of the relations on standard sequences \cite{SLC74,reutenauer}.

\begin{example}[\cite{VJM}]
For $\phi$ is the constant map and equals $1$, one has
$$\begin{array}{|c|c|c|}
\hline
l&\Pi_l&\Sigma_l\\
\hline
y_2&y_2-\frac{1}{2}y_1^2&y_2\\
y_1^2&y_1^2&\frac{1}{2}y_2+y_1^2\\
y_3&y_3-\frac{1}{2}y_1y_2-\frac{1}{2}y_2y_1+\frac{1}{3}y_1^3&y_3\\
y_2y_1&y_2y_1-y_2y_1&\frac{1}{2}y_3+y_2y_1\\
y_1y_2&y_2y_1-\frac{1}{2}y_1^3&v\\
y_1^3&y_1^3&\frac{1}{6}y_3+\frac{1}{2}y_2y_1+\frac{1}{2}y_1y_2+y_1^3\\
y_4&y_4-\frac{1}{2}y_1y_3-\frac{1}{2}y_2^2-\frac{1}{2}y_3y_1&y_4\\
&+\frac{1}{3}y_1^2y_2+\frac{1}{3}y_1y_2y_1+\frac{1}{3}y_2y_1^2-\frac{1}{4}y_1^4&\\
y_3y_1&y_3y_1-\frac{1}{2}y_2y_1^2-y_1y_3+\frac{1}{2}y_1^2y_2&\frac{1}{2}y_4+y_3y_1\\
y_2^2&y_2^2-\frac{1}{2}y_2y_1^2-\frac{1}{2}y_1^2y_2+\frac{1}{4}y_1^4&\frac{1}{2}y_4+y_2^2\\
y_2y_1^2&y_2y_1^2-2\,y_1y_2y_1+y_1^2y_2&\frac{1}{6}y_4+\frac{1}{2}y_3y_1+\frac{1}{2}y_2^2+y_2y_1^2\\
y_1y_3&y_1y_3-\frac{1}{2}y_1^2y_2-\frac{1}{2}y_1y_2y_1+\frac{1}{3}y_1^4&y_4+y_3y_1+y_1y_3\\
y_1y_2y_1&y_1y_2y_1-y_1^2y_2&\frac{1}{2}y_4+\frac{1}{2}y_3y_1+y_2^2\\
&&+y_2y_1^2+\frac{1}{2}y_1y_3+y_1y_2y_1\\
y_1^2y_2&y_1^2y_2-\frac{1}{2}y_1^4&\frac{1}{2}y_4+y_3y_1+y_2^2+y_2y_1^2\\
&&+y_1y_3+y_1y_2y_1+y_1^2y_2\\
y_1^4&y_1^4&\frac{1}{24}y_4+\frac{1}{6}y_3y_1+\frac{1}{4}y_2^2+\frac{1}{2}y_2y_1^2\\
&&+\frac{1}{6}y_1y_3+\frac{1}{2}y_1y_2y_1+\frac{1}{2}y_1^2y_2+y_1^4\\
\hline
\end{array}$$
\end{example}

\begin{remark}\label{R1}
\begin{enumerate}
\item Any letter in $Y$ is a Lie polynomial and primitive for $\Delta_{\shuffle}$ (see Remark \ref{R0}). Moreover, both $\ncp{\calL ie_A}{Y}$ and $\calP_{\shuffle}^Y$ are, by definition, closed by linear combinations and by the Lie bracket. Then, by \eqref{Bialgebras2}, one deduces
$\mathcal{U}(\calP_{\shuffle}^Y)=\mathcal{U}(\ncp{\calL ie_A}{Y})\cong\calH_{\shuffle}(Y)$.

\item For any $k>1$, by (\ref{Dstuffle}), since $\Delta_{{\shuffle}_{\phi}}y_k\neq y_k\otimes1_{Y^*}+1_{Y^*}\otimes y_k$ then one has
$\calP_{{\shuffle}_{\phi}}^Y\neq\ncp{\calL ie_A}{Y}$.
Then let $Y':=\{y'_k\}_{k\ge1}$ such that, for any $k\ge1$, one put $y'_k:=\pi_1(y_k)$.
On the one hand, by the previous item, $\ncp{\calL ie_A}{Y'}=\calP_{\shuffle}^{Y'}$ and, on the other hand, by Theorem \ref{isomorphy}, one has
$\calP_{{\shuffle}_{\phi}}^Y\cong\calP_{\shuffle}^{Y'}\cong\mathrm{Im}\,\pi_1$.
\end{enumerate}
\end{remark}

\begin{proposition}[MRS\footnote{MSR is an abbreviation of G. M\'elan\c con, M.P. Sch\"utzen\-berger and C. Reutenauer.} factorization]\label{diagonal}
On the bialgebras $\calH_{\shuffle}(\calX)$ and $\calH_{{\shuffle}_{\phi}}(Y)$, with the dual bases $\{S_w\}_{w\in\calX^*}$--$\{P_w\}_{w\in\calX^*}$ and $\{\Sigma_w\}_{w\in Y^*}$--$\{\Pi_w\}_{w\in Y^*}$ and the decreasing lexicographical order, the series ${\calD}_{\calX}$ and ${\calD}_Y$ are factorized as follows\footnote{Definition \ref{phishuffle}, Theorem \ref{isomorphy}, \eqref{Bialgebras1}--\eqref{dualbialgebras2} and Proposition \ref{diagonal} generalize results on $\stuffle$-bialgebra \cite{VJM}.}
\begin{eqnarray*}
{\calD}_{\calX}=\prod_{l\in\Lyn\calX}^{\searrow}e^{S_l\otimes P_l}
&\mbox{and}&
{\calD}_Y=\prod_{l\in\Lyn Y}^{\searrow}e^{\Sigma_l\otimes\Pi_l}.
\end{eqnarray*}
\end{proposition}

\begin{proof}
By Definition \ref{diaser} and the constructions in \eqref{P_w}--\eqref{Sigma_w}, one has
\begin{eqnarray*}
\calD_{\calX}=&\displaystyle\sum_{w\in\calX^*}S_w\otimes P_w&=
\sum_{l_1,\ldots,l_k\in\Lyn Y\atop l_1\succ\ldots\succ l_k}\frac{S_{l_1}^{\shuffle i_1}\shuffle\ldots\shuffle S_{l_k}^{\shuffle i_k}}{i_1!\ldots i_k!}\otimes P_{l_1}^{i_1}\ldots P_{l_k}^{i_k}\cr
\calD_Y=&\displaystyle\sum_{w\in Y^*}\Sigma_w\otimes\Pi_w
&=\sum_{l_1,\ldots,l_k\in\Lyn Y\atop l_1\succ\ldots\succ l_k}\frac{\Sigma_{l_1}^{{\shuffle}_{\phi}i_1}{\shuffle}_{\phi}\ldots{\shuffle}_{\phi}\Sigma_{l_k}^{{\shuffle}_{\phi}i_k}}{i_1!\ldots i_k!}\otimes\Pi_{l_1}^{i_1}\ldots\Pi_{l_k}^{i_k}.
\end{eqnarray*}
By \eqref{tensor1}--\eqref{tensor2}, it follows then the expected results.
\end{proof}

\begin{corollary}\label{factorisations}
Let $S\in\ncs{A}{\calX}$ (resp. $\ncs{A}{Y}$) be grouplike for $\Delta_{\shuffle}$ (resp. $\Delta_{{\shuffle}_{\phi}}$). Then\footnote{$\{\scal{S}{S_l}\}_{l\in\Lyn\calX}$ (resp. $\{\scal{S}{\Sigma_l}\}_{l\in\Lyn Y}$) are called local coordinates of second kind of $S$ in $\calG_{\shuffle}^{\calX}$ (resp. $\calG_{{\shuffle}_{\phi}}^{Y}$).}, using the decreasing lexicographical order, one has
\begin{eqnarray*}
S=\prod_{l\in\Lyn\calX}^{\searrow}e^{\scal{S}{S_l}P_l}
&\Big(\mbox{resp.}&
S=\prod_{l\in\Lyn Y}^{\searrow}e^{\scal{S}{\Sigma_l}\Pi_l}\Big).
\end{eqnarray*}
\end{corollary}

\begin{remark}\label{orders_Rk}
For various  applications, in extension of the results in \cite{PMB,Linz,ACA,CM}, $x_i$ (resp. $y_k$) can also be an alphabet $\{x_{k,i}\}_{k\ge1}$ (resp. $\{y_{x,k}\}_{x\in X}$), equipped the following total order, additionally to the orders in \eqref{orders},
\begin{eqnarray*}
&&x_{n,i}\succ\cdots\succ x_{n,i}\succ x_{n+1,i}\succ\cdots,\mbox{ for }1\le i\le m\\
&\mbox{(resp.}&
y_{x_0,k}\prec\cdots\prec y_{x_m,k},\mbox{ for }k\ge1),
\end{eqnarray*}
for which the set of Lyndon words over the alphabet $x_i$ (resp. $y_k$) is denoted by $\Lyn(x_k)$ (resp. $\Lyn(y_k)$) such that
\begin{eqnarray*}
&&x_0\prec\Lyn(x_1)\prec\cdots\prec\Lyn(x_m)\\
&\mbox{(resp.}&
\Lyn(y_1)\succ\cdots\Lyn(y_n)\succ\Lyn(y_{n+1})\succ\cdots).
\end{eqnarray*}
\end{remark}

\section{Algebraic combinatorics on representative series}\label{representativeseries}
\subsection{Various characterizations of representative series}
Representative (or rational) series are the representative functions
on the free monoid (see Section \ref{intro}).

\begin{definition}\label{dec1}
Let $S\in\ncs{A}{\calX}$ (resp. $\ncp{A}{\calX}$) and $P\in\ncp{A}{\calX}$ (resp. $\ncs{A}{\calX}$).
Then the {\it left} and the {\it right} \textit{shifts}\footnote{These are called {\it residuals} and extend shifts of functions in harmonic analysis \cite{jacob}.

In terms of representative functions, these correspond the {\it left} and {\it right} \textit{translates} \cite{abe,CartierPatras}.} of $S$ by $P$, $P\rd S$ and {$S\rg P$}, is defined as follows
\begin{eqnarray*}
\forall w\in\calX^*,&
\scal{{P\rd S}}{w}=\scal{S}{wP},&\scal{{S\rg P}}{w}=\scal{S}{Pw}.
\end{eqnarray*}
\end{definition}

\begin{remark}
The shifts operators are associative and mutually commute, \textit{i.e.}
\begin{eqnarray*}
S\rg(P\rg R)=(S\rg P)\rg R,&
P\rd(R\rd S)=(P\rd R)\rd S,&
(P\rg S)\rd R=P\rg(S\rd R)
\end{eqnarray*}
and then one has $x\rd(wy)=(yw)\rg x=\delta_{x,y}w$, for
$x,y\in\calX$ and $w\in\calX^*$.
\end{remark}

\begin{definition}[\cite{berstel}]\label{rational}
The series $S$ is rational (or representative) if it belongs to the smalest algebraic closure by rational operations (conatenation, addition,  Kleene star) containing $\ncp{A}{\calX}$.
The $A$-module of rational series is denoted by ${\ncs{A^{\mathrm{rat}}}\calX}$.
\end{definition}

\begin{definition}\label{dec1bis}
With the notations in \eqref{Bialgebras1}--\eqref{dualbialgebras2}, if $A=K$ being a field then one defines the Sweedler's dual $\calH_{\shuffle}^{\circ}(\calX)$ (resp. $\calH_{{\shuffle}_{\phi}}^{\circ}(Y)$) of $\calH_{\shuffle}(\calX)$ (resp. $\calH_{{\shuffle}_{\phi}}(Y)$), for $S\in\ncs{K}{\calX}$ (resp. $\ncs{K}{Y}$), as follows
\begin{eqnarray*}
S\in\calH_{\shuffle}^{\circ}(\calX)\mbox{ (resp. $\calH_{{\shuffle}_{\phi}}^{\circ}(Y)$)}
&\iff&\Delta_{\conc}(S)=\sum_{i\in I_f}G_i\otimes D_i,
\end{eqnarray*}
where $I_f$ is some finite set and $\{G_i,D_i\}_{i\in I_f}$ are series in $\calH_{\shuffle}^{\circ}(\calX)$ (resp. $\calH_{{\shuffle}_{\phi}}^{\circ}(Y)$).
\end{definition}

\begin{remark}\label{dec2}
Let $S\in\ncs{A}{\calX}$ and suppose that there is some finite set $I_f$  and a double family $\{G_i,D_i\}_{i\in I_f}$ of series in $\ncs{A}{\calX}$ such that, using $\Delta_{\conc}$ in \eqref{Dconc},
\begin{eqnarray*}
\Delta_{\conc}(S)=\sum_{i\in I_f}G_i\otimes D_i.
\end{eqnarray*}

Then, for any $v\in\calX^*$ and $i\in I_f$, putting $G'_i=G_i\rg v$ and $D'_i=v\rd D_i$, one has
\begin{enumerate}
\item $\Delta_{\conc}(S\rg v)=\sum\limits_{i\in I_f}G'_i\otimes D_i$ and
$\Delta_{\conc}(v\rd S)=\sum\limits_{i\in I_f}G_i\otimes D'_i$.

\item $\{S\rg v\}_{v\in\calX^*}$ (resp. $\{v\rd S\}_{v\in \calX^*}$) lie in a finitely generated shift-invariant $A$-module if and only if $\{G_i\rg v\}_{v\in \calX^*}$ (resp. $\{v\rd D_i\}_{v\in \calX^*}$) does (for $i\in I_f$).

\item If $S\in\calH_{\shuffle}^{\circ}(\calX)$ then $v\rd S$ and $S\rg v\in\calH_{\shuffle}^{\circ}(\calX)$ ($v\in\calX^*$).

\item $\calH_{\shuffle}(\calX)$ and $\calH_{{\shuffle}_{\phi}}(Y)$ are graded while $\calH_{\shuffle}^{\circ}(\calX)$ and $\calH_{{\shuffle}_{\phi}}^{\circ}(Y)$ are, generally, not.
\end{enumerate}
\end{remark}

\begin{theorem}\label{KS}
A series $S$ is rational if and only one of the following assertions holds
\begin{enumerate}
\item The shifts $\{S\trl w\}_{w\in\calX^*}$ (resp. $\{w\rd S\}_{w\in \calX^*}$) lie in a finitely generated shift-invariant $A$-module \cite{jacob-these}.

\item There is $n\in\N$ and a triplet $(\nu,\mu,\eta)$, so-called linear representation of dimension\footnote{The dimension is here (as in \cite{berstel}) the size of the matrices.} $n$ of $S$, where \cite{berstel}
\begin{eqnarray*}
\nu\in\calM_{n,1}(A),&\mu:\calX^*\longrightarrow\calM_{n,n}(A),&\eta\in\calM_{1,n}(A)
\end{eqnarray*}
such that, for any $w\in\calX^*$, one has (Kleene-Sch\"utzenberger theorem)
\begin{eqnarray*}
\scal{S}{w}=\nu\mu(w)\eta.
\end{eqnarray*}
\end{enumerate}
\end{theorem}

\begin{definition}\label{nil_sol}
\begin{enumerate}
\item Let $\calL$ be the Lie algebra. Then $\calL$ is said to be
\textit{nilpotent} (resp. \textit{solvable}) if and only if there exists an integer $k\ge1$ such that the sequence $\{\calL^n\}_{n\ge1}$ (resp. $\{\calL^{(n)}\}_{n\ge1}$), defined recursively as follows
\begin{eqnarray*}
\calL^1=\calL,\calL^{n+1}=[\calL,\calL^n]
&\mbox{(resp.}&
\calL^{(1)}=\calL,\calL^{(n+1)}=[\calL^{(n)},\calL^{(n)}])\end{eqnarray*}
satisfies $\calL^{k+1}=\{0\}$ (resp. $\calL^{(k+1)}=\{0\}$).

\item Let $(\nu,\mu,\eta)$ be a linear representation of $S\in\ncs{A^{\mathrm{rat}}}{\calX}$.
One defines
\begin{enumerate}
\item the Lie algebra generated by $\{\mu(x)\}_{x\calX}$ and denoted by $\calL(\mu)$,
\item the following function on monoid \begin{eqnarray*}
M:\calX^*&\longrightarrow&M_{n,n}(\ncs{A}{\calX}),\\
w&\longmapsto&\mu(w)w.
\end{eqnarray*}
\end{enumerate}
\end{enumerate}
\end{definition}

One has the following constructions of representations of representative series.
\begin{proposition}\label{linearrepresentation}
The module ${\ncs{A^{\mathrm{rat}}}\calX}$ (resp. ${\ncs{A^{\mathrm{rat}}}{Y}}$)
is closed by $\shuffle$ (resp. ${\shuffle}_{\phi}$). Moreover, for any $i=1,2$, let
$R_i\in{\ncs{A^{\mathrm{rat}}}\calX}$ and $(\nu_i,\mu_i,\eta_i)$ be its
representation of dimension $n_i$. Then the linear representation
$$\begin{array}{rccl}
\mbox{that of}&R_i^*&\mbox{is}&\Big(\begin{pmatrix}0&1\end{pmatrix},
\left\{\begin{pmatrix}\mu_i(x)+\eta_i\nu_i\mu_i(x)&0\cr\nu_i\eta_i&0\end{pmatrix}\right\}_{x\in\calX},
\begin{pmatrix}\eta_i\cr1\end{pmatrix}\Big),\cr
\mbox{that of}&R_1+R_2&\mbox{is}&\Big(\begin{pmatrix}\nu_1&\nu_2\end{pmatrix},
\left\{\begin{pmatrix}\mu_1(x)&{\bf 0}\cr{\bf 0}&\mu_2(x)\end{pmatrix}\right\}_{x\in\calX},
\begin{pmatrix}\eta_1\cr\eta_2\end{pmatrix}\Big),\cr
\mbox{that of}&R_1R_2\quad&\mbox{is}&\Big(\begin{pmatrix}\nu_1&0\end{pmatrix},
\left\{\begin{pmatrix}\mu_1(x)&\eta_1\nu_2\mu_2(x)\cr0&\mu_2(x)\end{pmatrix}\right\}_{x\in\calX},
\begin{pmatrix}\eta_1\mu_2\eta_2\cr\eta_2\end{pmatrix}\Big),\cr
\mbox{that of}&R_1\shuffle R_2&\mbox{is}&(\nu_1\otimes\nu_2,\{\mu_1(x)\otimes\mathrm{I}_{n_2}
+\mathrm{I}_{n_1}\otimes\mu_2(x)\}_{x\in\calX},\eta_1\otimes\eta_2),\\
\mbox{that of}&R_1{\shuffle}_{\phi} R_2&\mbox{is}&(\nu_1\otimes\nu_2,
\{\mu_1(y_k)\otimes\mathrm{I}_{n_2}+\mathrm{I}_{n_1}\otimes\mu_2(y_k)\cr
&&&+\sum\limits_{i+j=k}\gamma_{i,j}^k\mu_1(y_i)\otimes\mu_2(y_j)\}_{k\ge1},\eta_1\otimes\eta_2).
\end{array}$$
\end{proposition}

\begin{proof}
Only the last one is new and the first ones are treated in \cite{jacob}, see also \cite{DFLL}. The linear representations of $R_1\shuffle R_2$ and $R_1{\shuffle}_{\phi} R_2$ base on coproducts and tensor products of linear representations\end{proof}

\begin{example}[\cite{PMB}]
Let $Y=\{y_k\}_{k\ge1}$. A linear representations of dimension $1$ of $(t^ky_k)^*$ and $(-t^ky_k)^*$, for $k\ge1$, are respectively $(1,t^k,1)$ and $(1,-t^k,1)$.
For $\phi$ equals $1$, with the notations in Example \ref{q}, one has $(t^ky_k)^*\stuffle(-t^ky_k)^*=(-t^{2k}y_{2k})^*$ (see \cite{PMB})
and $(-t^{2k}y_{2k})^*$ admit $(1,-t^{2k},1)$ as linear representation of dimension $1$.
\end{example}

\begin{example}[\cite{PMB}]\label{automate}
The series $(t^2x_0x_1)^*$ and $(-t^2x_0x_1)^*$ admit a linear representation, respectively,
$(\nu_1,\{\mu_1(x_0),\mu_1(x_1)\},\eta_1)$ and $(\nu_2,\{\mu_2(x_0),\mu_2(x_1)\},\eta_2)$, where
$$\begin{array}{rlrl}
\nu_1=\begin{pmatrix}1&0\end{pmatrix},
&\mu_1(x_0)=\begin{pmatrix}0&t\cr0&0\end{pmatrix},
&\mu_1(x_1)=\begin{pmatrix}0&0\cr t&0\end{pmatrix},
&\eta_1=\begin{pmatrix}1\cr0\end{pmatrix},\cr
\nu_2=\begin{pmatrix}1&0\end{pmatrix},
&\mu_2(x_0)=\begin{pmatrix}0&\mathrm{i}t\cr0&0\end{pmatrix},
&\mu_2(x_1)=\begin{pmatrix}0&0\cr\mathrm{i}t&0\end{pmatrix},
&\eta_2=\begin{pmatrix}1\cr0\end{pmatrix}.
\end{array}$$
One has $(-t^2x_0x_1)^*\shuffle(t^2x_0x_1)^*=(-4t^4x_0^2x_1^2)^*$ (see also \cite{PMB}) and $(-4t^4x_0^2x_1^2)^*$ admits $(\nu,\{\mu(x_0),\mu(x_1)\},\eta)$ as a linear representation, where
$$\nu=\begin{pmatrix}1&0&0&0\end{pmatrix},
\mu(x_0)=\begin{pmatrix}
0&\mathrm{i}t&t&0\cr
0&0&0&t\cr
0&0&0&\mathrm{i}t\cr
0&0&0&0
\end{pmatrix},\\
\mu(x_1)=\begin{pmatrix}
0&0&0&0\cr
\mathrm{i}t&0&0&0\cr
t&0&0&0\cr
0&t&\mathrm{i}t&0
\end{pmatrix},
\eta=\begin{pmatrix}1\cr0\cr0\cr0\end{pmatrix}.$$
\end{example}

\begin{corollary}[Factorization and decomposition, \cite{CM}]\label{factorized}
Let $(\nu,\mu,\eta)$ be a linear representation of $S\in\ncs{A^{\mathrm{rat}}}{\calX}$. Then, with the notations of Definition \ref{nil_sol}, one has 
\begin{enumerate}
\item $S=\nu M(\calX^*)\eta$ and
\begin{eqnarray*}
M(\calX^*)=\prod_{l\in\Lyn\calX}^{\searrow}e^{\mu(P_ l)S_l}
&\Big(\mbox{resp.}&
M(Y^*)=\prod_{l\in\Lyn Y}^{\searrow}e^{\mu(\Pi_ l)\Sigma_l}\Big).
\end{eqnarray*}

\item If $\{M(x)\}_{x\in\calX}$ are upper triangular  then $S=\nu((D(\calX^*)N(\calX))^*D(\calX^*)\eta$, where $N(\calX)$ (resp. $D(\calX)$) is a strictly upper triangular (resp. diagonal) matrix such that $M(\calX)=N(\calX)+D(\calX)$. Moreover, $D(\calX^*)$ is diagonal and there is a positive interger $k$ such that $D(\calX^*)N(\calX)$ is nilpotent of order $k$ and then
$S=\nu(I_n+D(\calX^*)N(\calX^*)+\ldots+(D(\calX^*)N(\calX^*))^k)D(\calX^*)\eta$.
\end{enumerate}
\end{corollary}

\begin{proof}
\begin{enumerate}
\item By \eqref{diagonal}, one has  $M(\calX^*)=
(\mathrm{Id}\otimes\mu)\calD_{\calX}$ (resp. $M(Y^*)=(\mathrm{Id}\otimes\mu)\calD_Y$).

\item By Lazard factorization \cite{lothaire}, it follows the expected results.
\end{enumerate}
\end{proof}

\begin{proposition}[\cite{CM}]\label{PQ}
With the notations of Theorem \ref{KS}, for any $1\le i\le n$, let 
$e_i\in\calM_{1,n}(A)$ and ${}^te_i=\begin{matrix}(0&\ldots&1\!\!\!_{_{_{\displaystyle\uparrow i}}}&\ldots&0)\end{matrix}$, $G_i$ (resp $D_i$) belong to $\ncs{A^{\mathrm{rat}}}{\calX}$ admitting a linear representation of dimension $n$ $(\nu,\mu,e_i)$ (resp. $({}^te_i,\mu,\eta)$).
Then
\begin{eqnarray*}
\Delta_{\conc}S=\sum_{1\le i\le n}G_i\otimes D_i.
\end{eqnarray*}
\end{proposition}

\begin{proof}
Let us formulate the proof given in \cite{CM}, for any $u$ and $v\in\calX^*$, as follows
\begin{eqnarray*}
\scal{S}{uv}=\beta\mu(u)\mu(v)\eta=\sum_{1\le i\le n}(\nu\mu(u)e_i)({}^te_i\mu(v)\eta)=\sum_{1\le i\le n}\scal{G_i}{u}\scal{D_i}{v},\\
\scal{\Delta_{\tt conc}S}{u\otimes v}=\scal{S}{uv}=\sum_{1\le i\le n}\scal{G_i}{u}\scal{D_i}{v}=\sum_{1\le i\le n}\scal{G_i\otimes D_i}{u\otimes v}.
\end{eqnarray*}
Hence, extending by linearity, for any $P$ and $Q$ in $\ncp{A}{\calX}$, one obtains
\begin{eqnarray*}
\scal{S}{PQ}=\sum_{1\le i\le n}\scal{G_i}{P}\scal{D_i}{Q}
\end{eqnarray*}
and, using $\Delta_{\conc}$ in \eqref{Dconc}, it follows then the expected result.
\end{proof}

\begin{corollary}\label{Sweedler}
\begin{enumerate}
\item With the notations of Propositions \ref{linearrepresentation}--\ref{PQ}, one has
\begin{enumerate}
\item ${\ncs{A^{\mathrm{rat}}}\calX}$ is an unital $A$-algebra with respected to one of $\{\conc,\shuffle,{\shuffle}_{\phi}\}$.

\item The following criterion characterizes rational (or representative) series
\begin{eqnarray*}
S\in\ncs{A^{\mathrm{rat}}}{\calX}\iff
\Delta_{\conc}S=\sum_{i\in I_f}G_i\otimes D_i.
\end{eqnarray*}
\end{enumerate}

\item With the notations in Definition \ref{dec1bis}, the Sweedler's dual $\calH_{\shuffle}^{\circ}(\calX)$ (resp. $\calH_{{\shuffle}_{\phi}}^{\circ}(Y)$) is isomorphic to the bialgebra $(\ncs{K^{\mathrm{rat}}}{\calX},\shuffle,1_{\calX^*},\Delta_{\conc})$ (resp. $(\ncs{K^{\mathrm{rat}}}{Y},{\shuffle}_{\phi},1_{Y^*},\Delta_{\conc})$) of rational (or representative) series.
\end{enumerate}
\end{corollary}

\begin{proof}
\begin{enumerate}
\item These are consequences of Propositions \ref{linearrepresentation}--\ref{PQ}, respectively.

\item For $A=K$, the previous criterion and Definition \ref{dec1bis} lead to the expected results (see also Remark \ref{dec2}). 
\end{enumerate}
\end{proof}

\subsection{Kleene stars of the plane}
\begin{definition}\label{cylindric}
\begin{enumerate}
\item For any $(\alpha_x)_{x\in\calX}$ be a family of coefficients in $A$.
Then the following series is called a plane
\begin{eqnarray*}
L=\sum\limits_{x\in\calX}\alpha_xx&\in&\widehat{A.\calX}
\end{eqnarray*}
and the series $R=L^*$ is called Kleene star of the plane.

\item Any series $S\in\ncs{A}{{\calX}}$ is said to be
\begin{enumerate}
\item\label{syntacticallyexchangeable} syntactically exchangeable if and only if it is constant on multi-homo\-geneous classes, \textit{i.e.}
\begin{eqnarray*}
(\forall u,v\in\calX^*)([(\forall x\in\calX)(|u|_x=|v|_x)]&\Rightarrow&\scal{S}{u}=\scal{S}{v}).
\end{eqnarray*}
The set of syntactically exchangeable series is denoted by $\ncs{A_{\mathrm{exc}}^{\mathrm{synt}}}{\calX}$.

\item\label{rationallylyexchangeable} rationally exchangeable if and only if it admits a representation $(\nu,\mu,\eta)$
such that the matrices $\{\mu(x)\}_{x\in \calX}$ commute and the set of these series, a shuffle
subalgebra of $\ncs{A}{X}$, is denoted by $\ncs{A_{\mathrm{exc}}^{\mathrm{rat}}}{\calX}$.
\end{enumerate}
\end{enumerate}
\end{definition}

\begin{remark}
\begin{enumerate}
\item $S$ is syntactically exchangeable if and only if it is of the form\footnote{Recall that $\calX$ could be infinite (see \eqref{alphabets}) and the support of the map $\alpha:\calX\longrightarrow\N$ is finite.}
\begin{eqnarray*}
S=\sum_{\alpha\in\N^{(\calX)},\mathrm{supp}(\alpha)=\{x_1,\ldots,x_k\}}
s_{\alpha}x_1^{\alpha(x_1)}\shuffle\ldots\shuffle x_k^{\alpha(x_k)}.
\end{eqnarray*}

\item For $A=K$ is a field, the rational exchangeable series (Definition \ref{cylindric}.\ref{rationallylyexchangeable}) are exactly those that admit a representation with commuting matrices (at least the minimal one is such) and it is taken as definition as, even for rings, implying syntactic exchangeability (Definition \ref{cylindric}.\ref{syntacticallyexchangeable}). 
\end{enumerate}
\end{remark}

\begin{theorem}[See \cite{Ngo2,CM}]\label{exchangeable}
\begin{enumerate}
\item\label{ratex1} In all cases, one has
\begin{eqnarray*}
\ncs{A_{\mathrm{exc}}^{\mathrm{rat}}}{\calX}
\subset\ncs{A^{\mathrm{rat}}}{\calX}\cap\ncs{A_{\mathrm{exc}}^{\mathrm{synt}}}{\calX}.
\end{eqnarray*}
The equality holds when $A$ is a field and, letting
\begin{eqnarray*}
\ncs{A_{\mathrm{fin}}^{\mathrm{rat}}}{Y}
=\bigcup\limits_{F\subset_{finite}Y}\ncs{A^{\mathrm{rat}}}{F},
\end{eqnarray*}
the algebra of series over finite subalphabets, one has\footnote{The last inclusion of Item \ref{ratex1} is strict as shows the example of the following identity, living in $\ncs{A_{\mathrm{exc}}^{\mathrm{rat}}}{Y}$ but not in $\ncs{A_{\mathrm{exc}}^{\mathrm{rat}}}{Y}\cap \ncs{A_{\mathrm{fin}}^{\mathrm{rat}}}{Y}$
\begin{eqnarray*}
(ty_1+t^2y_2+\ldots)^*
=\lim\limits_{k\rightarrow+\infty}(ty_1+\ldots+t^ky_k)^*
=\lim\limits_{k\rightarrow+\infty}(ty_1)^*\shuffle\ldots\shuffle(t^ky_k)^*
=\mathop{\Huge\shuffle}_{k\ge1}(t^ky_k)^*.
\end{eqnarray*}}
\begin{eqnarray*}
\ncs{A_{\mathrm{exc}}^{\mathrm{rat}}}{X}
&=&\mathop{\Huge\shuffle}\limits_{x\in X}\ncs{A^{\mathrm{rat}}}{x},\\
\ncs{A_{\mathrm{exc}}^{\mathrm{rat}}}{Y}\cap\ncs{A_{\mathrm{fin}}^{\mathrm{rat}}}{Y}
&=&\bigcup\limits_{k\ge0}\mathop{\Huge\shuffle}\limits_{j=1}^k\ncs{A^{\mathrm{rat}}}{y_j}
\subsetneq\ncs{A_{\mathrm{exc}}^{\mathrm{rat}}}{Y}.
\end{eqnarray*}

\item\label{Kronecker}
For any $x\in\calX$, one has
$\ncs{A^{\mathrm{rat}}}{x}=\{P(1-xQ)^{-1}\}_{P,Q\in A[x]}$
and if $A=K$ is an algebraically closed field of characteristic zero then\footnote{It can be rephrased in terms of stars as
$\ncs{A^{\mathrm{rat}}}{x}=\{P(xQ)^*\}_{P,Q\in A[x]}$ holds for any ring and is then characteristic free, unlike the shuffle version requiring algebraic closure and denominators.
}
\begin{eqnarray*}
\ncs{K^{\mathrm{rat}}}{x}=\span_K\{(ax)^*\shuffle\ncp{K}{x}\vert a\in K\}.
\end{eqnarray*}

\item\label{conccharacter} Any Kleene star of the plane is $\conc$-character and conversely.

\item\label{indepchargen}
$A$ is supposed without zero divisors. If the family $(\phi_i)_{i\in I}$is $\Z$-linearly independent within $\widehat{A\calX}$ then the
family $\Lyn\calX\uplus\{\phi_i^*\}_{i\in I}$ is $A$-algebraically free within $(\ncs{A^{\mathrm{rat}}}{\calX},\shuffle,1_{\calX^*})$. 

\item\label{indepchar} In particular, $\{x^*\}_{x\in\calX}$ (resp. $\{y^*\}_{y\in Y}$) are algebraically independent over $(\ncp{A}{\calX},\shuffle,1_{\calX^*})$ (resp. $(\ncp{A}{Y},{\shuffle}_{\phi},1_{Y^*})$)
within $(\ncs{A^{\mathrm{rat}}}{\calX},\shuffle,1_{\calX^*})$
(resp. $(\ncs{A^{\mathrm{rat}}}{Y},{\shuffle}_{\phi},1_{Y^*})$).
\end{enumerate}
\end{theorem}

\begin{proof}
\begin{enumerate}
\item The inclusion is obvious in view of Definition \ref{cylindric}. For the equality, it suffices to prove that, when $A$ is a field,  every rational and exchangeable series admits a representation with commuting matrices. This is true of any minimal representation as shows the computation of shifts (see \cite{DR,Ngo2,CM}). Now, if $\calX$ is finite, then (all matrices commute, see also Corollary \ref{factorized})
\begin{eqnarray*}
M(\calX^*)=\mathop{\Huge\shuffle}_{x\in\calX}(\mu(x)x)^*
\end{eqnarray*}
and the result comes from the fact that $R$ is a linear combination of matrix elements. As regards the second equality, inclusion $\supset$ is straightforward. We remark that $\cup_{k\ge1}\ncs{A^{\mathrm{rat}}}{y_1}\shuffle\ldots\shuffle\ncs{A^{\mathrm{rat}}}{y_k}$ is directed as these algebras are nested in one another. With this in view, the reverse inclusion comes from the fact that every
$S\in\ncs{A_{\mathrm{fin}}^{\mathrm{rat}}}{Y}$ is a series over a finite alphabet and the result follows from the first equality.

\item This is a Kronecker's theorem \cite{zygmund}, rephrased in terms and notations of \cite{berstel}.

\item With the notations in Definition \ref{rationallylyexchangeable}, if $R$ is a Kleene star of the phan then $\scal{R}{1_{\calX^*}}=1_A$. Furthermore, if $w=xu$ then
$\scal{R}{xu}=\alpha_x\scal{R}{u}$. Thus, by induction on the length,
$\scal{R}{x_1\ldots x_k}=\alpha_{x_1}\times\cdots\times\alpha_{x_k}$
showing that $R$ is a $\conc$-character. Conversely (Sch\"utzenberger's reconstruction lemma)
\begin{eqnarray*}
R=\scal{R}{1_{\calX^*}}.1_A+\sum_{x\in\calX}x.x^{-1}R.
\end{eqnarray*}
But, if $S$ is a $\conc$-character, \textit{i.e.}
$\scal{R}{1_{\calX^*}}=1$ and $x^{-1}R=\scal{R}{x}R$,
then the previous expression reads, as follows, proving the claim
\begin{eqnarray*}
R=1_A+\Big(\sum_{x\in\calX}\scal{R}{x}x\Big)R,&\mbox{or equivalently,}&
R=\Big(\sum_{x\in\calX}\scal{R}{x}x\Big)^*.
\end{eqnarray*}

\item As $(\ncp{A}{\calX},\shuffle)$ and $(\ncp{A}{Y},{\shuffle}_{\phi})$ are enveloping algebras, this property is an application of the fact that, on an enveloping  $\mathcal{U}$, the characters are linearly independent with respect to the convolution algebra $\mathcal{U}^*_{\infty}$ (see the general construction and proof in \cite{DGM}. Here, this  convolution algebra ($\mathcal{U}^*_{\infty}$) contains the polynomials (is equal in case of finite $\calX$). Now, consider a monomial
\begin{eqnarray*}
(\psi_{i_1}^*)^{\shuffle\alpha_1}\ldots(\psi_{i_n}^*)^{\shuffle\alpha_n}=(\alpha_{i_1}\psi_{i_1}+\cdots+\alpha_{i_n}\psi_{i_n})^*.
\end{eqnarray*}
The $\Z$-linear independence of the monomials in $(\psi_i)_{i\in I}$ implies that all these monomials are linearly independent over $\ncp{A}{\calX}$ which proves algebraic independence of the family $(\psi_i)_{i\in I}$.

To end with, the fact that $\Lyn\calX\uplus\{\psi_i^*\}_{i\in I}$ is algebraically free comes from the transitivity of polynomial algebras (see \cite{B_Alg1}) and a Radford's theorem, \textit{i.e.} $(\ncp{A}{\calX},\shuffle,1_{\calX^*})\simeq A[\Lyn\calX]$.

\item Comes directly as an application of the preeceding point.
\end{enumerate}
\end{proof}

\begin{proposition}\label{simple}
Let $\alpha_x,\beta_x,a_s,b_s$ be complex numbers ($x\in\calX$ and $s\ge1$). Then\footnote{By \eqref{explog}, one has $\exp_{\shuffle}A\exp_{\shuffle}B=\exp_{\shuffle}(A+B)$ but it is more complicated for ${\shuffle}_{\phi}$.}
\begin{eqnarray*}
\Big(\sum_{x\in\calX}\alpha_xx\Big)^*\shuffle\Big(\sum_{x\in\calX}\beta_xx\Big)^*
&=&\Big(\sum_{x\in\calX}(\alpha_x+\beta_x)x\Big)^*,\cr
\Big(\sum_{s\ge1}a_sy_s\Big)^*{\shuffle}_{\phi}\Big(\sum_{s\ge1}b_sy_s\Big)^*
&=&\Big(\sum_{s\ge1}(a_s+b_s)y_s+\sum_{r,s\ge1}a_sb_r\phi(y_s,y_r)\Big)^*.
\end{eqnarray*}
\end{proposition}

\begin{proof}
Let us use $\Delta_{\shuffle}$ (resp. $\Delta_{{\shuffle}_{\phi}}$) defined in (\ref{Dshuffle}) and then (\ref{D2}) (resp. (\ref{Dstuffle}) and then (\ref{D1})) and, for any $x_i\in\calX,y_t\in Y$, apply respectively (\ref{extenoverseries2}) and (\ref{extenoverseries3})
\begin{eqnarray*}
&&\scal{\Big(\sum_{x\in\calX}\alpha_xx\Big)^*\shuffle\Big(\sum_{x\in\calX}\beta_xx\Big)^*}{x_i}\cr
&&=\scal{\Big(\sum_{x\in\calX}\alpha_xx\Big)^*\otimes\Big(\sum_{x\in\calX}\beta_xx\Big)^*}{\Delta_{\shuffle}(x_i)}\cr
&&=\scal{\Big(\sum_{x\in\calX}\alpha_xx\Big)^*\otimes\Big(\sum_{x\in\calX}\beta_xx\Big)^*}{x_i\otimes1_{X^*}+1_{X^*}\otimes x_i}\cr
&&=\alpha_i+\beta_i\cr
&&=\scal{\Big(\sum_{x\in\calX}(\alpha_x+\beta_x)x\Big)^*}{x_i},\\
&&\scal{\Bigl(\sum_{s\ge1}a_sy_s\Bigr)^*{\shuffle}_{\phi}\Bigl(\sum_{s\ge1}b_sy_s\Bigr)^*}{y_t}\cr
&&=\scal{\Bigl(\sum_{s\ge1}a_sy_s\Bigr)^*\otimes\Bigl(\sum_{s\ge1}b_sy_s\Bigr)^*}{\Delta_{{\shuffle}_{\phi}}(y_t)}\cr
&&=\scal{\Bigl(\sum_{s\ge1}a_sy_s\Bigr)^*\otimes\Bigl(\sum_{s\ge1}b_sy_s\Bigr)^*}
{y_t\otimes1_{Y^*}+1_{Y^*}\otimes y_t+\sum_{r+s=t}\gamma_{s,r}^ty_s\otimes y_r}\cr
&&=a_t+b_t+\sum_{r+s=t}a_sb_r\\
&&=\scal{\Bigl(\sum_{s\ge1}(a_s+b_s)y_s+\sum_{r,s\ge1}a_sb_r\phi(y_s,y_r)\Bigr)^*}{y_t}.
\end{eqnarray*}
\end{proof}

\begin{example}
For $\shuffle_{\phi}=\stuffle$, since
\begin{eqnarray*}
\sum_{r,s\ge1}a_sb_ry_{s+r}=
\sum_{r>s\ge1}(a_sb_r+a_rb_s)y_{s+r}+\sum_{s\ge1}a_sb_sy_{2s}
\end{eqnarray*}
then, by Proposition \ref{simple}, one also has
\begin{eqnarray*}
\big(\sum_{s\ge1}a_sy_s\big)^*\stuffle\big(\sum_{s\ge1}b_sy_s\big)^*
=\big(\sum_{s\ge1}(a_s+b_s)y_s+a_sb_sy_{2s}
+\sum\limits_{r>s\ge1}(a_sb_r+a_rb_s)y_{s+r}\big)^*.
\end{eqnarray*}
In particular, one has (for $y_s,y_r\in Y$ and $a_s,a_r\in\C$)
\begin{eqnarray*}
(a_sy_s)^*\stuffle(a_ry_r)^*&=&(a_sy_s+a_ry_r+a_sa_ry_{s+r})^*,\\
(-a_sy_s)^*\stuffle(a_sy_s)^*&=&(-a_s^2y_{2s})^*.
\end{eqnarray*}
Then, for any $c\in\C$ and $y_k\in Y$, one has
\begin{eqnarray*}
(cy_k^*)^{\stuffle 2}=&(cy_k)^*\stuffle(cy_k)^*&=(2cy_k + c^2y_{2k})^*,\\
(cy_k^*)^{\stuffle 3}=&(2cy_k+c^2y_{2k})^*\stuffle (cy_k)^*&
=(3cy_k+3c^2y_{2k}+c^3y_{3k})^*.
\end{eqnarray*}
\end{example}

\begin{corollary}\label{star}
\begin{enumerate}
\item\label{deux} Let $k,n\in\N,c\in\C$ and $x\in\calX$. Then
\begin{eqnarray*}
\scal{(cx)^*\shuffle(1+cx)^n}{x^k}={n+k\choose k}c^k.
\end{eqnarray*}
\item\label{trois} Let $x\in\calX,y_k\in Y$ and $c\in\C,n\in\N_{\ge1}$. One has\footnote{\leftline{$((cx)^*)^n=\underbrace{(cx)^*\ldots(cx)^*}_{\mbox{$n$ times}}, ((cx)^*)^{\shuffle n}=\underbrace{(cx)^*\shuffle\ldots\shuffle(cx)^*}_{\mbox{$n-1$ times $\shuffle$}}, ((cy_k)^*)^{\stuffle n}=\underbrace{(cx)^*\stuffle\ldots\stuffle(cx)^*}_{\mbox{$n-1$ times $\stuffle$}}$.}}
\begin{eqnarray*}
((cx)^*)^{\shuffle n}=(ncx)^*,\quad
((cx)^*)^n=(cx)^*\shuffle(1+cx)^{n-1},\cr
((cy_k)^*)^{\stuffle n}=\Big(\sum_{i=1}^n\dbinom{n}{i}c^i y_{ik}\Big)^*=\mathop{\Huge\shuffle}_{i=1}^n\Big(\dbinom{n}{i}c^i y_{ik}\Big)^*.
\end{eqnarray*}
\item\label{quatre}
For any $k,m\in\N,x_1,\ldots,x_m\in\calX$ and $l_1,\ldots,l_m\in\N, l_1+\ldots+l_m=k$, let
\begin{eqnarray*}
P_k:=x_1^{l_1}\shuffle\cdots\shuffle x_m^{l_m}&\mbox{and}&L_k:=\mathrm{supp}(P_k).
\end{eqnarray*}
For any $n_1,\ldots,n_m\in\N,c,c_1,\ldots,c_m\in\C\setminus\{0\}$ and $w\in L_k$, one has\footnote{Recall that, for any positive integer $k$ and nonnegative integers
$m,n_1,n_2,\cdots,n_m$, the generalized Chu-Vandermonde's identity is expressed as follows \cite{stanley}
\begin{eqnarray*}
\sum_{l_1+l_2+\cdots+l_m=k}{n_1\choose l_1}{n_2\choose l_2}\cdots{n_m\choose l_m}={n_1+n_2+\cdots+n_m\choose k}.
\end{eqnarray*}},
\begin{eqnarray*}
\scal{\mathop{\Huge\shuffle}_{i=1}^m((c_ix_i)^*)^{n_i+1}}{w}
&=&\sum_{l_1+\ldots+l_m=k}{n_1+l_1\choose l_1}\cdots{n_m+l_m\choose l_m}
c_1^{l_1}\cdots c_m^{l_m},\cr
\scal{\mathop{\Huge\shuffle}_{i=1}^m((cx_i)^*)^{n_i+1}}{w}
&=&{n_1+\cdots+n_m+k\choose k}c^k.
\end{eqnarray*}
\end{enumerate}
\end{corollary}

\begin{proof}
\begin{enumerate}
\item By (\ref{extenoverseries2}), with
\begin{eqnarray*}
S=(cx)^*,&R=(1+cx)^n,&(cx)^k\shuffle(cx)^i={k+i\choose k}(cx)^{k+i},
\end{eqnarray*}
one has
\begin{eqnarray*}
(cx)^*\shuffle(1+cx)^n
=\sum_{k\ge0}\sum_{i=0}^n{k+i\choose k}{n\choose i}(cx)^{k+i}
=\sum_{k\ge0}\Big(\sum_{i=0}^n{k\choose k-i}{n\choose i}\Big)(cx)^k
\end{eqnarray*}
and the Chu-Vandermonde identity \cite{stanley} yields the expected result.

\item Since $(cx)^*=\exp_{\shuffle}(cx)$ then
$((cx)^*)^{\shuffle n}=\exp_{\shuffle}(ncx)=(ncx)^*$.
The two last identities are obvious for $n=1$ and supposed to hold, up to rank $n\ge1$. Next, by
\begin{eqnarray*}
((cx)^*)^{n+1}=(cx)^*((cx)^*)^n&\mbox{and}&
((cy_k)^*)^{\stuffle n+1}=(cy_k)^*\stuffle((cy_k)^*)^{\stuffle n}
\end{eqnarray*}
and then by induction hypothesis, one obtains
\begin{eqnarray*}
((cx)^*)^{n+1}
&=&(cx)^*((cx)^*)^n\cr
&=&(cx)^*\Big(\sum_{k\ge0}{n-1+k\choose k}(cx)^k\Big)\mbox{ (by Item \ref{deux})}\\
&=&\sum_{k\ge0}\Big(\sum_{l=0}^k{n-1+l\choose l}\Big)(cx)^k\mbox{ (by (\ref{extenoverseries1}))}\cr
&=&\sum_{k\ge0}{n+k\choose n}(cx)^k\mbox{ (by the Chu-Vandermonde identity)}\\
&=&(cx)^*\shuffle(1+cx)^n\mbox{ (by Item \ref{deux})},\\
(cy_k^*)^{\stuffle n+1}
&=&\Big(cy_k +\sum_{i=1}^n{n\choose i}c^iy_{ik}+\sum_{i=1}^n{n\choose i}c^{i+1}y_{(i+1)k}\Big)^*\cr
&=&\Big(cy_k+{n\choose 1}c y_k+\sum_{i=2}^n{n\choose i}c^iy_{ik}+ \sum_{i=2}^{n+1}{n\choose i-1}c^{i} y_{ik}\Big)^*\cr
&=&\Big({n+1\choose 1}c y_k+\sum_{i=2}^n\Big({n\choose i}+{n\choose i-1}\Big)c^i y_{ik}+{n+1\choose n+1}c^{n+1}y_{(n+1)k}\Big)^*\cr
&=&\Big({n+1\choose 1}cy_k+\sum_{i=2}^n{n+1\choose i}c^iy_{ik}+{n+1\choose n+1}c^{n+1}y_{(n+1)k}\Big)^*\cr
&=&\Big(\sum_{i=1}^{n+1}{n+1\choose i}c^i y_{ik}\Big)^*\cr
&=&\mathop{\Huge\shuffle}_{i=1}^n\Big({n\choose i}c^i y_{ik}\Big)^*
\mbox{ (by Proposition \ref{simple})}.
\end{eqnarray*}

\item By Items \ref{deux}--\ref{trois} and Proposition \ref{simple} , one gets
\begin{eqnarray*}
\mathop{\Huge\shuffle}_{i=1}^m((c_ix_i)^*)^{n_i+1}
&=&\mathop{\Huge\shuffle}_{i=1}^m(c_ix_i)^*\shuffle(1+c_ix_i)^{n_i}\\
&=&\mathop{\Huge\shuffle}_{i=1}^m
\sum_{l_i\ge0}{n_i+l_i\choose l_i}(c_ix_i)^{l_i}\cr
&=&\sum_{k\ge0}\sum_{l_1+\ldots+l_m=k}
{n_1+l_1\choose l_1}\cdots{n_m+l_m\choose l_m}c_1^{l_1}\cdots c_m^{l_m}P_k.
\end{eqnarray*}
It follows then the expected results.
\end{enumerate}
\end{proof}

\begin{corollary}[Kleene stars of the plane]\label{Kleene}
Let $R$ and $L\in\ncs{A^{\mathrm{rat}}}{\calX}$ such that
\begin{eqnarray*}
L^*=R,&\scal{R}{1_{\calX^*}}=1,&\scal{L}{1_{\calX^*}}=0.
\end{eqnarray*}
Then the following assertions are equivalent.
\begin{enumerate}
\item\label{caracter} $R$ is a $\conc$-character of $(\ncp{A}{\calX},\conc,1_{\calX^*})$.
\item\label{exp_rat} $R$ is a Kleene star of the plane.
\item\label{lin_rep} $R$ admits a linear representation of dimension $1$.
\item\label{plane} $L$ is a plane.
\item\label{log} $L$ is an infinitesimal $\conc$-character of $(\ncp{A}{\calX},\conc,1_{X^*})$.
\end{enumerate}
\end{corollary}

\begin{proof}
{\ref{caracter} $\iff$ \ref{exp_rat}:}
This corresponds to the item \ref{conccharacter} of Theorem \ref{exchangeable}.

{\ref{exp_rat} $\iff$ \ref{lin_rep}}
This is a direct consequence of Theorem \ref{KS}.

{\ref{exp_rat} $\iff$ \ref{plane}:}
This is obvious, by construction (in which $L$ is the $\shuffle$-logarithm of $R$). Indeed, as in Definition \ref{exchangeable} and \eqref{explog}, one has
\begin{eqnarray*}
R=\mathop{\Huge\shuffle}_{x\in\calX}(\alpha_xx)^*=\mathop{\Huge\shuffle}_{x\in\calX}\exp_{\shuffle}(\alpha_xx)=\exp_{\shuffle}L.
\end{eqnarray*}

{\ref{plane} $\iff$ \ref{log}:}
If $L$ is an infinitesimal character then, by Definition \ref{dec0}, one obtains $\scal{L}{uv}=\scal{L}{u}\scal{v}{1_{\calX^*}}+\scal{u}{1_{\calX^*}}\scal{L}{v}$, for $u$ and $v\in\calX^*$.
Hence, for $w=uv\in\calX^{\ge2}$ with $u$ and $v\in\calX^+$, one has $\scal{L}{w}=\scal{L}{uv}=0$.

In addition, $\scal{L}{1_{\calX^*}}=0$ (for $u=v=1_{\calX^*}$) and it follows that
\begin{eqnarray*}
L=\sum_{x\in\calX}\scal{L}{x}x.
\end{eqnarray*}

Conversely, since $\scal{uv}{x}=\scal{u}{x}\scal{v}{1_{\calX^*}}+\scal{u}{1_{\calX^*}}\scal{v}{x}=0$, for $x\in\calX$ and $u,v\in\calX^+$,
then by \eqref{pairing}, one deduces that $\scal{L}{uv}=0$ meaning that $L$ is an infinitesimal $\conc$-character
\begin{eqnarray*}
\scal{L}{uv}=\sum\limits_{x\in\calX}\scal{L}{x}\scal{uv}{x}=0.
\end{eqnarray*}
\end{proof}

\begin{remark}\label{reetheorem}
In Corollary \ref{Kleene}, if $A$ is a field $K$ then point \ref{caracter} (resp. \ref{log}) can be rephrased as ``$R$ is a grouplike series" (resp. ``$L$ is a primitive series") for $\Delta_{\conc}$.

Indeed, in \eqref{D1}--\eqref{D3}, if $S\in\ncs{K}{Y}$ (resp. $\ncs{K}{\calX}$ is a ${\shuffle}_{\phi}$ (resp. $\shuffle,\conc$)-character of  $(\ncp{K}{Y},\conc,1_{Y^*})$
(resp. $(\ncp{K}{\calX},\conc,1_{\calX^*})$ then
\begin{enumerate}
\item Since $S\otimes S=\sum\limits_{u,v\in\calX^*}\scal{S}{u}\scal{S}{v}u\otimes v$ then $\Delta_{{\shuffle}_{\phi}}(S)=S\otimes S$ (resp. $\Delta_{\shuffle}(S)=S\otimes S$ and $\Delta_{\conc}(S)=S\otimes S$).

\item Since $\Delta_{{\shuffle}_{\phi}}$ (resp. $\Delta_{\shuffle}$ and $\Delta_{\conc}$), the maps $T\longmapsto T\otimes1_{Y^*}$ and $T\longmapsto1_{Y^*}\otimes T$
(resp. $T\longmapsto T\otimes1_{\calX^*}$ and $T\longmapsto1_{\calX^*}\otimes T$)
are continuous homomorphisms, then\footnote{Here, $\log S\otimes1_{Y^*}$ and $1_{Y^*}\otimes\log S$
(resp. $\log S\otimes1_{\calX^*}$ and $1_{\calX^*}\otimes\log S$) commute.}
$\Delta_{{\shuffle}_{\phi}}(\log S)=\log S\otimes1_{\calX^*}+1_{\calX^*}\otimes\log S$
(resp. $\Delta_{\shuffle}(\log S)=\log S\otimes1_{\calX^*}+1_{\calX^*}\otimes\log S$
and $\Delta_{\conc}(\log S)=\log S\otimes1_{\calX^*}+1_{\calX^*}\otimes\log S$).
\end{enumerate}
Then $S$ is grouplike series for $\{\Delta_{{\shuffle}_{\phi}},\Delta_{\shuffle},\Delta_{\conc}\}$, if and only if $\log S$ is primitive meaning that the equivalence, between the items \ref{caracter} and \ref{log}, is an extension of the Ree's theorem (see Proposition \ref{proposition2} and Remark \ref{Ree}.).
\end{remark}

To establish Corollary \ref{Subbialgebras} below, we will use a Lie's theorem \cite{serre}, held over $A=K$ as an algebraically closed field of characteristic $0$ \cite{serre}, and the following lemma.

\begin{lemma}\label{lemma}
Let $(\nu,\tau,\eta)$ a representation of $S$ of dimension $r$ such that, for all 
$x\in \calX$, ($\tau(x)-c(x)I_r$) is strictly upper triangular, then 
$S\in\ncs{K_{\mathrm{exc}}^{\mathrm{rat}}}{\calX}\shuffle\ncp{K}{\calX}$.
\end{lemma}

\begin{proof}
Let $(e_i)_{1\le i\le r}$ be the canonical basis of $M_{1,r}(K)$.
We construct the representations of $S_1$ and $S_2$
\begin{eqnarray*}
\rho_1=(\nu,(x\longmapsto\tau(x)-c(x)I_r),\eta)&\mbox{and}& 
\rho_2=(e_1,(x\longmapsto c(x)I_r),{}^te_1)
\end{eqnarray*}
and remark that $S_1\shuffle S_2$ admits the representation 
\begin{eqnarray*}
\rho_3=(\nu\otimes e_1,((\tau(x)-c(x)I_r)\otimes I_r+I_r\otimes c(x)I_r)_{x\in\calX},\eta\otimes{}^te_1)
\end{eqnarray*}
as $I_r\otimes c(x)I_r=c(x)I_r\otimes I_r$, $\rho_3$ is, in fact, $(\nu\otimes e_1,(\tau(x)\otimes I_r)_{x\in \calX},\eta\otimes{}^te_1)$
which represents $S$, the result now comes from the fact that $S_1\in\ncp{K}{\calX}$
and $S_2=\Big(\sum\limits_{x\in \calX}c(x)x\Big)^*\in\ncs{K_{\mathrm{exc}}^{\mathrm{rat}}}{\calX}$.
\end{proof}

\begin{corollary}[Triangular sub bialgebras of $(\ncs{K^{\mathrm{rat}}}{\calX},\shuffle,1_{X^*},\Delta_{\conc})$]\label{Subbialgebras}
Let $(\nu,\mu,\eta)$ be a linear representation of $R\in\ncs{K^{\mathrm{rat}}}{\calX}$. Then, with notations in Definition \ref{nil_sol},
\begin{enumerate}
\item\label{CommRep} $\calL(\mu)$ is commutative if and only if $R\in\ncs{K_{\mathrm{exc}}^{\mathrm{rat}}}{\calX}$,

\item\label{NilRep} $\calL(\mu)$ is nilpotent if and only if 
$R\in{\ncs{K_{\mathrm{exc}}^{\mathrm{rat}}}{\calX}}\shuffle\ncp{K}{\calX}$,

\item\label{SollRep} $\calL(\mu)$ is solvable if and only if $R$ is a linear combination of series in the form $E_1x_{i_1}\ldots E_jx_{i_j}E_{j+1}$, where $E_1,\ldots,E_{j+1}\in\serie{K_{\mathrm{exc}}^{\mathrm{rat}}}{\calX}$ and $x_{i_1},\cdots,x_{i_j}\in\calX$.
\end{enumerate}
\end{corollary}

\begin{proof}
\begin{enumerate}
\item Since $\calL(\mu)$ is commutative then, due to the commutation of matrices, \textit{i.e.} $\scal{R}{pxys}=\scal{R}{pyxs}$, for $x,y\in\calX$ and $p,s\in \calX^*$. Conversely, since $\rho$ is minimal then there is $P_i$ and $Q_i\in\ncp{K}{\calX}$ ($i=1...n$) such that, for any $u\in\calX^*$, (see \cite{berstel,DR})
\begin{eqnarray*}
\mu(u)=(\scal{P_i\trr R\trl Q_i}{u})_{1\leq i,j\leq n}=(\scal{R}{Q_iuP_i})_{1\leq i,j\leq n}.
\end{eqnarray*}
Now, for any $x,y\in\calX$,
\begin{eqnarray*}
\mu(xy)=(\scal{R}{Q_ixyP_i})_{1\leq i,j\leq n}\stackrel{(*)}{=}(\scal{R}{Q_iyxP_i})_{1\leq i,j\leq n}=\mu(yx),
\end{eqnarray*}
(equality $\stackrel{(*)}{=}$ being due to exchangeability).

\item Since $\calL(\mu)$ is nilpotent then let $K^n$ be the space of the representation of $\calL(\mu)$ and $\bigoplus\limits_{j=1}^m V_j$ be a decomposition of $K^n$ into indecomposable $\calL(\mu)$-modules (see \cite{dixmier} for characteristic $0$, or \cite{Lie7} for arbitrary characteristic), we know that $V_j$ is an $\calL(\mu)$-module and the action of $\calL(\mu)$ is triangularisable with constant diagonals inside each sector $V_j$. Thus, it is an invertible matrix $P$ in $\mathrm{GL}(n,K)$ such that, for any $x\in \calX$,
\begin{eqnarray*}
P\mu(x)P^{-1}=\mathrm{blockdiag}(T_1,T_2\ldots,T_k)=
\begin{pmatrix}
T_1&0&\ldots&0\cr
0&T_2&\ldots&0\cr
\vdots&\vdots&\ddots&\vdots\cr
0&0&\ldots&T_k
\end{pmatrix},
\end{eqnarray*}
where the $T_j$'s are upper triangular with the constant coefficients in the diagonal. Let $d_j$ be the dimension of $T_j$ (so that $n=\sum\limits_{j=1}^md_j$) and partitioning  $\nu'=\nu P^{-1}$ (resp. $\eta'=P\eta$) with these dimensions, we get blocks so that each $(\nu'_j,T_j,\eta'_j)$ is the linear representation of the series $R_j$ and $R=\sum\limits_{j=1}^mR_j$. By Lemma \ref{lemma}, each $R_j$ belongs to $\ncs{K_{\mathrm{exc}}^{\mathrm{rat}}}{\calX}\shuffle\ncp{K}{\calX}$ and so is their sum $R$.

Conversely, if $\rho_i=(\nu_i,\tau_i,\eta_i),i=1,2$, are two representations then
\begin{eqnarray*}
&&[\tau_1(x)\otimes I_r+I_r\otimes\tau_2(x),\tau_1(y)\otimes I_r+I_r\otimes\tau_2(y)]\cr
&&=[\tau_1(x)\otimes I_r,\tau_1(y)\otimes I_r]
+[\tau_1(x)\otimes I_r,I_r\otimes\tau_2(y)]\cr
&&+[I_r\otimes\tau_2(x),\tau_1(y)\otimes I_r]
+[I_r\otimes\tau_2(x),I_r\otimes\tau_2(y)]\cr
&&=[\tau_1(x),\tau_1(y)]\otimes I_r+I_r\otimes [\tau_2(x),\tau_2(y)]
\end{eqnarray*}
because
\begin{eqnarray*}
[\tau_1(x)\otimes I_r,\tau_1(y)\otimes I_r]
&=&\tau_1(xy)\otimes I_r-\tau_1(yx)\otimes I_r\cr
&=&[\tau_1(x),\tau_1(y)]\otimes I_r,\cr
[\tau_1(x)\otimes I_r,I_r\otimes\tau_2(y)]
&=&\tau_1(x)\otimes\tau_2(y)-\tau_2(y)\otimes\tau_1(x)\cr
&=&0,\cr
[I_r\otimes\tau_2(x),\tau_1(y)\otimes I_r]
&=&\tau_1(y)\otimes\tau_2(x)-\tau_1(y)\otimes\tau_2(x)\cr
&=&0,\cr
[I_r\otimes\tau_2(x),I_r\otimes\tau_2(y)]
&=&I_r\otimes\tau_2(xy)-I_r\otimes\tau_2(yx)\cr
&=&I_r\otimes[\tau_2(x),\tau_2(y)].
\end{eqnarray*}
Similar formula hold for $m$-fold brackets (Dynkin combs), so that if $\calL(\tau_i)$'s are nilpotent, the Lie algebra $\calL(\tau_1\otimes I_r+I_r\otimes\tau_2)$ is also nilpotent.
The point here comes from the fact that series in $\ncs{K_{\mathrm{exc}}^{\mathrm{rat}}}{\calX}$ as well as in $\ncp{K}{\calX}$ admit nilpotent representations, so, let $(\alpha,\tau,\beta)$ such a representation and $(\alpha',\tau',\beta')$ its minimal quotient (obtained by minimization, see \cite{berstel}), then $\calL(\tau')$ is nilpotent as a quotient of $\calL(\tau)$. Now two minimal representations being isomorphic, $\calL(\mu)$ is isomorphic to $\calL(\tau)$ and then it is nilpotent.

\item Since $\calL(\mu)$ is solvable then, by a Lie's theorem \cite{serre}, $\{\mu(x)\}_{x\in X}$ are simultaneously upper triangularisable.
In the favorable change of bases, without loss generality, one suppose that the linear representation $(\nu,\mu,\eta)$ of $R$ is such that
each $\mu(x)$ is upper triangular and it follows that $M(x)$ is also upper triangular and $D(\calX^*)N(\calX)$ is nilpotent of order $k$.
Let ${\calS}$ be the vector space generated by the expected form of series. Then $\calS$ is closed by concatenation and then $D(\calX^*)N(\calX)$
and $(D(\calX^*)N(\calX))^*$ belong to ${\calS}^{n\times n}$. Finally, by Corollary~\ref{factorized}, $R=\nu M(\calX^*)\eta\in{\calS}$.

Conversely, as sums and quotients of solvable representations  are solvable, it suffices to show that a single expected form admits a solvable representation and end by quotient and isomorphism as in (\ref{NilRep}). By Proposition~\ref{linearrepresentation}, if $R_i$ admits solvable representations so does $R_1R_2$, then the claim follows from the fact that, firstly, single letters admit solvable (even nilpotent) representations and secondly series of $\mathop{\shuffle}\limits_{x\in\calX}\{\ncs{K^{\mathrm{rat}}}{x}\}$ admit solvable representations. Finally, we choose (or construct) a solvable representation of $R$, call it $(\alpha,\tau,\beta)$ and  $(\alpha',\tau' \beta')$ its minimal quotient, then $\calL(\tau')$ is solvable as a quotient of $\calL(\tau)$. Now two minimal representations being isomorphic, $\calL(\mu)$ is isomorphic to $\calL(\tau)$, hence solvable.
\end{enumerate}
\end{proof}

\begin{remark}
\begin{enumerate}
\item For an example of representative series with solvable representation but such that $S\notin\ncs{K_{\mathrm{exc}}^{\mathrm{rat}}}{\calX}\shuffle\ncp{K}{\calX}$, one takes $S=x_0^*x_1(-x_0)^*$ (for $\calX=\{x_0,x_1\}$) admitting the following linear representation.
\begin{eqnarray*}
\nu=\begin{pmatrix}0&1\end{pmatrix},
&\mu(x_0)=
\begin{pmatrix}
-1&0\cr
0&1
\end{pmatrix}
\mbox{ and }
\mu(x_1)=
\begin{pmatrix}
0&1\cr
0&0
\end{pmatrix},
&\eta=\begin{pmatrix}1\cr0\end{pmatrix}.
\end{eqnarray*}

\item Denoting by $\ncs{K_{\mathrm{nil}}^{\mathrm{rat}}}{\calX}$ (resp. $\ncs{K_{\mathrm{sol}}^{\mathrm{rat}}}{\calX}$),
the set of representative series such that $\calL(\mu)$ is nilpotent (resp. solvable), we get a tower of sub Hopf algebras of the Sweedler's dual,
$\ncs{K_{\mathrm{nil}}^{\mathrm{rat}}}{\calX}\subset\ncs{K_{\mathrm{sol}}^{\mathrm{rat}}}{\calX}\subset\calH_{\shuffle}^{\circ}(\calX)$.
\end{enumerate}
\end{remark}

\section{Conclusion}
In this work, various products such as concatenation, shuffle and its $\phi$-deforma\-tions (\textit{i.e.} $\conc,\shuffle$ and ${\shuffle}_{\phi}$), and their coproducts $\{\Delta_{\conc},\Delta_{\shuffle},\Delta_{{\shuffle}_{\phi}}\}$, of noncommutative series, over a finite alphabet (as $X=\{x_0,\ldots,x_m\}$) or an infinite alphabet (as $Y=\{y_k\}_{k\ge1}$) and with coefficients in a ring $A$ containing $\Q$,  were examined.
Basing on various (effective) constructions of pair of dual bases in the graded bialgebras $(\ncp{A}{X},1_{X^*},\conc,\Delta_{\shuffle})$ or $(\ncp{A}{Y},1_{Y^*},\conc,\Delta_{\shuffle_{\phi}})$, the rational series viewed as representative functions on the free monoids $(X^*,1_{X^*})$ or $(Y^*,1_{Y^*})$, were factorized and decomposed within their associated bialgebras $(\ncs{A^{\mathrm{rat}}}{X},1_{X^*},\shuffle,\Delta_{\tt conc})$ or $(\ncs{A^{\mathrm{rat}}}{Y},{Y^*},{\shuffle}_{\phi},\Delta_{\tt conc})$.
In particular, for $A$ being a field $K$, these $K$-bialgebras of representative series are also proved to be isomorphic to the Sweedler's duals of the $K$-bialgebras of polynomials. 
Moreover, for the concatenation, only Kleene stars of the planes are characters, or equivalently (by use of a Ree's theorem like), only the planes are infinitesimal characters.

\section*{Appendix}\label{polylogarithms}
Representative series constitute, as already said in Section \ref{intro}, an algebraic combinatorics framework for calculations in computer algebra dealing with polylogarithms and harmonic sums indexed by words as well as by rational series, for which a synthesis of main results is being presented to serve readers and future works.

Let us consider again notations in Examples \ref{pos}--\ref{Hypergeometric} (and Remark \ref{multiindex}) and use the morphisms of algebras $\pi_Y:(A1_{X^*}\oplus\AX x_1,{\conc},1_{Y^*})\longrightarrow(\AY,{\conc},1_{Y^*})$, mapping $x_0^{s_1-1}x_1\ldots x_0^{s_r-1}x_1$ to $y_{s_1}\ldots y_{s_r}$.
Let $\calH(\Omega)$ be the ring of holomorphic functions on the simply connected domain $\Omega:=\widetilde{\C\setminus\{0,1\}}$, admitting $1_{\Omega}$ as the neutral element, and $\calC:=\C[z,z^{-1},(1-z)^{-1}]$ be the subring of $\calH(\Omega)$. Let $\calG$ denote the group of transformations, generated by $\{z\mapsto1-z,z\mapsto1/z\}$, permuting the singularities in $\{0,1,+\infty\}$. Then, for any $G\in\calC$ and $g\in\calG$, one has $G(g(z))\in\calC$.

\begin{theorem}[structure theorem, \cite{Ngo,FPSAC98,CM}]\label{structure}
One has
\begin{enumerate}
\item The following morphism are {\it injective} (in particular $\Li_{x_0^r}(z)={\log^r(z)}/{r!}$)
\begin{eqnarray*}
\H_{\bullet}:(\Q\pol{Y},\stuffle,1_{Y^*})\longrightarrow(\Q\{\H_w\}_{w\in Y^*},\times,1),&&w\longmapsto\H_w,\label{H}\cr
\Li_{\bullet}:(\QX,\shuffle,1_{X^*})\longrightarrow(\Q\{\Li_w\}_{w\in X^*},\times,1_{\Omega}),&&w\longmapsto\Li_w,\label{Li}
\end{eqnarray*}
Hence, $\{\H_w\}_{w\in Y^*}$ and $\{\Li_w\}_{w\in X^*}$ (resp. $\{\H_l\}_{l\in\Lyn Y}$ and $\{\Li_l\}_{l\in\Lyn X}$) are $\Q$-linearly (resp.algebraically) independent.

\item $\calC\{\Li_w\}_{w\in X^*}\cong\calC\otimes\C\{\Li_w\}_{w\in X^*}$, closed under the action of $\mathcal G$ and closed by $\theta_0=z\partial_z$ and $\theta_1=(1-z)\partial_z$ and their sections such that $\theta_0\iota_0=\theta_1\iota_1=\mathrm{Id}$.

\item There is also a law $\top$, in $\serie{\Q}{Y_0}$, such that the following morphisms are surjective (in particular, $\Li^-_{y_0^r}(z)=\big(z/(1-z)\big)^r, \H^-_{y_0^r}(n)=\binom{n}{r}=(n)_r/r!$).
\begin{eqnarray*}
\H^-_{\bullet}:(\Q\pol{Y_0}, \stuffle,1_{Y_0^*})\longrightarrow(\Q\{\H^-_w\}_{w\in Y_0^*},\times,1),&&w\longmapsto\H^-_w,\cr
\Li^-_{\bullet}:(\Q\pol{Y_0},\top,1_{Y_0^*})\longrightarrow(\Q\{\Li^-_w\}_{w\in Y_0^*},\times,1_{\Omega}),&&w\longmapsto\Li^-_w,
\end{eqnarray*}
Hence, $\{\H^-_{y_k}\}_{k\ge0}$ and $\{\Li^-_{y_k}\}_{k\ge0}$ are $\Q$-linearly free and $\ker\H^-_{\bullet}=\ker\Li^-_{\bullet}=\Q\langle\{w-w\top 1_{Y_0^*}\mid w\in Y_0^*\}\rangle$. Moreover, if $\top':{\Q}\pol{Y_0}\times{\Q}\pol{Y_0}\rightarrow{\Q}\pol{Y_0}$ is a law such that $\Li^-_{\bullet}$ is a morphism for $\top'$ and $(1_{Y_0^*}\top'{\Q}\pol{Y_0})\cap\ker(\Li^-_{\bullet})=\{0\}$ then $\top'=f\circ\top$, where $f\in GL({\Q}\pol{Y_0})$ is such that $\Li^-_{\bullet}\circ f=\Li^-_{\bullet}$.
\end{enumerate}
\end{theorem}

Using Theorem \ref{structure}, the following polymorphism is {\it surjective} (by definition)
\begin{eqnarray}\label{zeta}
\zeta:{\displaystyle(\Q1_{X^*}\oplus x_0\QX x_1,\shuffle,1_{X^*})\atop
\displaystyle(\Q1_{Y^*}\oplus(Y\setminus\{y_1\})\Q\pol{Y},\stuffle,1_{Y^*})}&\longtwoheadrightarrow&(\calZ,\times,1),\\
{\displaystyle x_0x_1^{s_1-1}\ldots x_0x_1^{s_k-1}\atop\displaystyle y_{s_1}\ldots y_{s_k}}&\longmapsto&\zeta(s_1,\ldots,s_k),
\end{eqnarray}
where $\calZ$ is the $\Q$-algebra generated by $\{\zeta(l)\}_{l\in\Lyn X\setminus X}$ (resp. $\{\zeta(S_l)\}_{l\in\Lyn X\setminus X}$),
or equivalently, generated by $\{\zeta(l)\}_{l\in\Lyn Y\setminus\{y_1\}}$ (resp. $\{\zeta(\Sigma_l)\}_{l\in\Lyn Y\setminus\{y_1\}}$).
The grouplike series $\H:=(\H_{\bullet}\otimes\mathrm{Id}_Y){\mathcal D}_Y$ and $\L:=(\Li_{\bullet}\otimes\mathrm{Id}_X){\mathcal D}_X$ are defined, as images of ${\mathcal D}_Y$ and ${\mathcal D}_X$ in \eqref{diagonal}, and allow defining the following grouplike series \cite{FPSAC98,VJM}
\begin{eqnarray}
\H=\prod_{l\in\Lyn Y}^{\searrow}e^{\H_{\Sigma_l}\Pi_l}
&\mbox{and set}&
Z_{\stuffle}:=\prod_{l\in\Lyn Y\setminus\{y_1\}}^{\searrow}e^{\zeta(\Sigma_l)\Pi_l},\\
\L=\prod_{l\in\Lyn X}^{\searrow}e^{\Li_{S_l}P_l}
&\mbox{and set}&
Z_{\shuffle}:=\prod_{l\in\Lyn X\setminus X}^{\searrow}e^{\zeta(S_l)P_l}.
\end{eqnarray}

Singularities analysing on coefficients of $\L$ \cite{FPSAC98} and, using an identity of the type Newton-Girard \cite{Daresbury,JSC}, one obtains successively
\begin{eqnarray}
\L(z)\sim_0\exp(x_0\log z)&\mbox{and}&\L(z)\sim_1\exp(-x_1\log(1-z))Z_{\shuffle},\\\label{globalasymptotic}
\H(n)\sim_{+\infty}\sum_{k\ge0}\H_{y_1^k}y_1^k\pi_YZ_{\shuffle}
&\mbox{and}&
\sum_{k\ge0}\H_{y_1^k}y_1^k=e^{\sum_{k\ge1}\H_{y_k}(n){(-y_1)^k}/k}.\label{asymptotic}
\end{eqnarray}

\begin{theorem}[\cite{Daresbury,JSC,VJM}]\label{pont}
Let us consider the following morphism of algebras
\begin{eqnarray*}
\gamma_{\bullet}:(\Q\pol{Y},\stuffle,1_{Y^*})&\longrightarrow&(\calZ[\gamma],\times,1),\\
w&\longmapsto&\gamma_w=\mathrm{f.p.}_{n\rightarrow+\infty}\H_w(n),\quad\{n^a\log^b(n)\}_{a\in\Z,b\in\N},
\end{eqnarray*}
and let $Z_{\gamma}$ be the noncommutative generating series of $\{\gamma_w\}_{w\in Y^*}$. Then one has
\begin{enumerate}
\item $\gamma$ is a $\stuffle$-character.
Hence, $Z_{\gamma}$ is group-like for $\Delta_{\stuffle}$ and $Z_{\gamma}=e^{\gamma y_1}Z_{\stuffle}$.

\item Setting $B(y_1):=e^{\gamma y_1-\sum_{k\ge2}{\zeta(k)}(-y_1)^k/{k}}$ and $B'(y_1):=e^{-\sum_{k\ge2}{\zeta(k)}(-y_1)^k/{k}}$,
one has the following global renormalization
\begin{eqnarray*}
\lim_{z\rightarrow 1}e^{y_1\log(1-z)}\pi_Y(\L(z))=\lim_{n\rightarrow\infty}e^{\sum_{k\ge1}\H_{y_k}(n){(-y_1)^k}/k}\H(n)=\pi_Y{Z}_{\shuffle}
\end{eqnarray*}
and then $Z_{\gamma}=B(y_1)\pi_Y{Z}_{\shuffle}$, or equivalently, $Z_{\stuffle}=B'(y_1)\pi_Y{Z}_{\shuffle}$.
\end{enumerate}
\end{theorem}

Now, since $\zeta_{\shuffle}(x_0)=\Li_{x_0}(1)=0$ then, using the facts that
\begin{eqnarray}
\mathrm{f.p.}_{z\rightarrow1}\Li_{x_1}(z)=0,&&\{(1-z)^a\log^b((1-z)^{-1})\}_{a\in\Z,b\in\N},\\
\mathrm{f.p.}_{n\rightarrow+\infty}\H_{y_1}(n)=0,&&\{n^a\H_1^b(n)\}_{a\in\Z,b\in\N},
\end{eqnarray}
the polymorphism $\zeta$ in \eqref{zeta} can be extended as characters as follows \cite{VJM,CM}
\begin{eqnarray}
&\zeta_{\shuffle}:(\QX,\shuffle,1_{X^*})\longrightarrow({\mathcal Z},\times,1),&
\zeta_{\stuffle}:({\Q\pol{Y}},\stuffle,1_{Y^*})\longrightarrow({\mathcal Z},\times,1),
\end{eqnarray}
according to products and satisfying, for generators of length (resp. weight) one, $\zeta_{\shuffle}(x_0)=\zeta_{\shuffle}(x_1)=\zeta_{\stuffle}(y_1)=0$ and $\zeta_{\shuffle}(l)=\zeta_{\stuffle}(\pi_Yl)=\gamma_{l}=\zeta(l)$, $l\in\Lyn X\setminus X$.

\begin{remark}\label{algorithm}
\begin{enumerate}
\item $\{\scal{Z_{\shuffle}}{u}\}_{u\in X^*}$ and $\{\scal{Z_{\stuffle}}{v}\}_{v\in Y^*}$ represent, respectively,
\begin{eqnarray*}
\scal{Z_{\shuffle}}{w}=\zeta_{\shuffle}(w)=&\mathrm{f.p.}_{z\rightarrow1}\Li_w(z),&\{(1-z)^a\log^b((1-z)^{-1})\}_{a\in\Z,b\in\N},\\
\scal{Z_{\stuffle}}{w}=\zeta_{\stuffle}(w)=&\mathrm{f.p.}_{n\rightarrow+\infty}\H_w(n),&\{n^a\H_1^b(n)\}_{a\in\Z,b\in\N}.
\end{eqnarray*}

\item Theorem \ref{pont} provides equations bridging stuctures of polyzetas and the algorithm {\bf LocalCoordinateIdentification} \cite{tokyo}, identying the locale coordinates in each equation, gives the kernels of $\zeta$ as $\shuffle$ or $\stuffle$ ideals (encoding the algebraic relations among $\{\zeta(S_l)\}_{l\in\Lyn X\setminus X}$ or among $\{\zeta(\Sigma_l)\}_{l\in\Lyn Y\setminus\{y_1\}}$, independant on the constant $\gamma$ \cite{VJM,CM}).
These ideals are totally lexicographically ordered and constituted by homogenous in weight polynomials. These are viewed as confluent rewriting systems in which, the left side of each rule is the leading term of the associated homogenous polynomial and the irreducible terms (of each system) encode an algebraic basis of $\calZ$ \cite{tokyo}.
\end{enumerate}
\end{remark}

Similarly, asymptotic behaviors of $\{\Li^-_w\}_{w\in Y^*_0},\{\H^-_w\}_{w\in Y^*_0}$ are analyzed by

\begin{theorem}[\cite{Ngo}]\label{renormalization2}
\begin{enumerate}
\item $\H^-_w$ and $\Li^-_w$ are polynomial, of degree $(w)+|w|$ in $\Q[n]$ and $\Z[(1-z)^{-1}]$, respectively, for $w\in Y_0^*,n\in\mathbb{N}_+,z\in\C,\abs{z}<1$. Hence,
\begin{eqnarray*}
\H^-_{w}(n)\sim_{+\infty}n^{(w)+\abs{w}}C^-_w
&\mbox{and}&
\Li^-_{w}(z)\sim_1(1-z)^{-(w)-\abs{w}}B^-_w,\\
C^-_w=\prod_{w=uv\atop v\neq 1_{Y_0^*}}((v)+\abs{v})^{-1}
&\mbox{and}&B^-_w=((w)+\abs{w})!C^-_w.
\end{eqnarray*}

\item Let $h(t):=\sum\limits_{w\in Y_0^*}{((w)+\abs w)!}{t^{(w)+\abs w}}w$ and $g(t):=\big(\sum\limits_{y\in Y_0}t^{(y)+1}y\big)^*$ and
\begin{eqnarray*}
\L^-:=\sum_{w\in Y_0^*}\Li^-_ww,\quad
\H^-:=\sum_{w\in Y_0^*}\H^-_ww,\quad
C^-:=\sum_{w\in Y_0^*}C^-_ww.
\end{eqnarray*}
Then the following global renormalization holds
\begin{eqnarray*}
\lim_{z\to1}h^{\odot-1}((1-z)^{-1})\odot\L^-(z) =\lim_{n\to+\infty}g^{\odot-1}(n)\odot\H^-(n)=C^-.
\end{eqnarray*}
\end{enumerate}
\end{theorem}

Any polynomial $p\in(\Z[t],\times,1)$ of degree $d$ is associated with the polynomials $\hat p\in(\Q[t],\times,1)$ and $\check p\in(\Z[t],\shuffle,1)$, of same degree, as follows
\begin{eqnarray}\label{LimoinsHmoins}
\check p(t)=\sum_{0\le i\le d}p_it^{\shuffle i},
&p(t)=\displaystyle\sum_{0\le i\le d}p_it^i,
&\hat p(t)=\sum_{0\le i\le d}\frac{p_i}{i!}t^i.
\end{eqnarray}

Let us consider
\begin{eqnarray}
(n+\bullet):\N\longrightarrow\Q,&&i\longmapsto(n+i)_{n}=\frac{(n+i)!}{n!},\\
\chi:(\Q[y_1^*],\stuffle,1_{Y^*})\longrightarrow(\Q[(n+\bullet)_n],\times,1),&&
S\longmapsto\H_S,
\end{eqnarray}
and let $\Q[(n+\bullet)_n]$ be the set of polynomials (on $n$) expanded as follows
\begin{eqnarray}\label{bullet}
\forall\pi\in\Q[(n+\bullet)_n],&\deg(\pi)=d,&\pi=\sum_{0\le i\le d}\pi_i(n+i)_n.
\end{eqnarray}

\begin{proposition}[extension of $\Li_{\bullet}$, \cite{Ngo,CM}]\label{explicit}
Let $d=\abs{w}+(w)$, for $w\in Y_0^*\setminus\{1_{Y_0^*}\}$. For any $n\in\mathbb{N}_+$ and $z\in\C,\abs{z}<1$, with the notations in \eqref{LimoinsHmoins}--\eqref{bullet},  one has
\begin{enumerate}
\item There is an unique polynomial $p\in(\Z[t],\times,1)$ of degree $d$ such that
\begin{eqnarray*}
\Li^-_w(z)=&\Sum_{0\le k\le d}\frac{p_k}{(1-z)^k}&=\sum_{0\le k\le d}p_ke^{-k\log(1-z)},\\
\H^-_w(n)=&\Sum_{0\le k\le d}p_k{n+k\choose k}&=\sum_{0\le k\le d}\frac{p_k}{k!}(n+k)_n.
\end{eqnarray*}

\item There is an unique polynomial $R_w$ of degree $d$ in $(\Z[x_1^*],\shuffle,1_{X^*})$ such that
\begin{eqnarray*}
\Li_{R_w}(z)=\Li^-_w(z)=&p(e^{-\log(1-z)})&\in(\Z[e^{-\log(1-z)}],\times,1_{\Omega}),\cr
\H_{\pi_Y(R_w)}(n)=\H^-_w(n)=&\hat p((n+\bullet)_n)&\in(\Q[(n+\bullet)_n],\times,1).
\end{eqnarray*}
For $w=y_{s_1},\ldots y_{s_r}$, one explicitly has $\Li^-_w(z)=\Li_{R_w}(z)=\alpha_0^z(R_w)$, where\footnote{The $S_1(k,i)$'s and $S_2(k,j)$'s denote the Stirling numbers of first and second kind, respectively.}
\begin{eqnarray*}
R_w
&=&\sum_{k_1=0}^{s_1}\sum_{k_2=0}^{s_1+s_2-k_1}\ldots
\sum_{k_r=0}^{(s_1+\ldots+s_r)-\atop(k_1+\ldots+k_{r-1})}
\binom{s_1}{k_1}\binom{s_1+s_2-k_1}{k_2}\cdots\cr
&&\cdots\binom{s_1+\ldots+s_r-k_1-\ldots-k_{r-1}}{k_r}\shuffle\limits_{1\le i\le r}\rho_{k_i},\quad\mbox{and}\cr
\rho_{k_i}&=&\left\{\begin{array}{lcl}
x_1^*-1_{X^*}&\mbox{if}&k_i=0,\cr
\displaystyle\sum_{j=1}^{k_i}S_2({k_i},j)(j!)^2\sum_{l=0}^j
\frac{(-1)^{l}}{l!}\frac{(x_1^*)^{\shuffle(j-l+1)}}{(j-l)!}&\mbox{if}&k_i\neq0.
\end{array}\right.
\end{eqnarray*}

\item Extending by linearity $R_{\bullet}$ over $\Q\pol{Y_0}$, for $y_k$ and $y_l\in Y_0$, one has
\begin{eqnarray*}
\Li_{R_{y_k}\shuffle R_{y_l}}=\Li_{R_{y_k}}\Li_{R_{y_l}}=\Li^-_{y_k}\Li^-_{y_l}=\Li^-_{y_k\top y_l}=\Li_{R_{y_k\top y_l}}.
\end{eqnarray*}

\item The restriction $\Li_{\bullet}:(\Z[x_1^*],\shuffle,1_{X^*})\longrightarrow(\Z[e^{-\log(1-z)}],.,1_{\Omega})$ is bijective and then so is $R_{\bullet}:({\Z}\langle{Y}\rangle,\top,1_{Y^*})\longrightarrow(\Z[x_1^*],\shuffle,1_{X^*})$.

\item For any $k\ge1$, one has $\Li_{-k}=\Li_{R_{y_k}}$, where
\begin{eqnarray*}
R_{y_k}=x_1^*\shuffle R'_{y_k}&\mbox{and}&
R'_{y_k}=\sum_{i=0}^{k}i!S_2(k,i)(x_1^*-1)^{\shuffle i}
\end{eqnarray*}
and, conversely,
\begin{eqnarray*}
R_{y_0}=x_1^*-1_{X^*}&\mbox{and}&
(kx_1)^*=1_{X^*}+R_{y_0}+\sum_{j=2}^k\frac{S_1(k,j)}{(k-1)!}R_{y_{j+1}}.
\end{eqnarray*}
\end{enumerate}
\end{proposition}

\begin{proposition}[extension of $\H_{\bullet}$, \cite{PMB}]\label{Indexation2}
For any $r\ge1,t\in\C,\abs{t}<1$, one has
\begin{eqnarray*}
\H_{(t^ry_r)^*}=\sum_{k\ge0}\H_{y_r^k}t^{kr}
=\exp\biggl(\sum_{k\ge1}\H_{y_{kr}}\frac{(-t^r)^{k-1}}k\biggr).
\end{eqnarray*}
Moreover, for any $\abs{a_s}<1$ and $\abs{b_s}<1$ and $\abs{a_s+b_s}<1$,
\begin{eqnarray*}
\H_{(\sum\limits_{s\ge1}a_sy_s)^*}\H_{(\sum\limits_{s\ge1}b_sy_s)^*}=\H_{(\sum\limits_{s\ge1}(a_s+b_s)y_s+\sum\limits_{r,s\ge1}a_sb_ry_{s+r})^*}.
\end{eqnarray*}
Hence, for any $r$ and $s\ge1$, one has
\begin{eqnarray*}
y_r^*=\exp_{\stuffle}\Bigl(\sum_{k\ge1}y_{kr}\frac{(-1)^{k-1}}k\Bigr)
&\mbox{and}&y_r^*\stuffle y_s^*=\sum_{k=0}^r{r+s-k\choose s}{s\choose k}y_{r+s-k}
\end{eqnarray*}
and then, for any $n\ge0$,
\begin{eqnarray*}
y_r^n=\frac{(-1)^n}{n!}\sum_{s_1,\ldots,s_n>0,s_1+\ldots+ns_n=n}
\frac{(-y_r)^{\stuffle s_1}}{1^{s_1}}\stuffle\ldots\stuffle\frac{(-y_{nr})^{\stuffle s_n}}{n^{s_n}}.
\end{eqnarray*}
\end{proposition}

Since $\Li_{(tx_1)^*}(z)=(1-z)^{-t}$ then, for any $t\in\C,\abs{t}<1$, by Propositions \ref{explicit}-\ref{Indexation2}, the characters $\zeta_{\shuffle}$ and $\gamma_{\bullet}$ can be, now, extended algebraically as follows \cite{PMB}
\begin{eqnarray}
\zeta_{\shuffle}:(\CX\shuffle\ncs{\C_{\mathrm{exc}}^{\mathrm{rat}}}{X},\shuffle,1_{X^*})\longrightarrow(\C,.,1),&&
{(tx_0)^*\atop(tx_1)^*}\longmapsto{1_{\C}},\\
\gamma_{\bullet}:(\CY\stuffle\ncs{\C_{\mathrm{exc}}^{\mathrm{rat}}}{Y},\stuffle,1_{Y^*})\longrightarrow(\C,.,1),&&
(t^ry_r)^*\longmapsto\Gamma_{y_r}^{-1}(1+t),
\end{eqnarray}
where, for any $r\ge1$ and $t\in\C,\abs{t}<1$, one put $\Gamma_{y_r}(1+t):=e^{-\ell_r(t)}$ and
\begin{eqnarray*}
\ell_1(z):=\gamma z-\sum_{k\ge2}\zeta(k)\frac{(-z)^k}k&\mbox{and}&
\ell_r(z):=-\sum_{k\ge1}\zeta(kr)\frac{(-z^r)^k}k,r\ge2.
\end{eqnarray*}

\begin{theorem}[\cite{PMB}]\label{plein}
Let $\C[L]$ and $\C[E]$ be the algebras of $L=\span_\C\{\ell_r\}_{r\ge1}$ and $E=\span_{\C}\{e^{\ell_r}\}_{r\ge1}$, respectively. Then one has
\begin{enumerate}
\item\label{free} $(\ell_r)_{r\ge1}$ and $(e^{\ell_r})_{r\ge1}$ are $\C$-linearly free and free from $1_{\calH(\Omega)}$.
\item $(\ell_r)_{r\ge1}$ and $(e^{\ell_r})_{r\ge1}$ are $\C$-algebraically independent.
\item For any $r\ge1$, one has
\begin{enumerate}
\item $\ell_r$ and $e^{\ell_r}$ are $\C$-algebraically independent.
\item $\ell_r$ is holomorphic on the open unit disc, $D_{<1}$,
\item $e^{\ell_r}$ (resp. $e^{-\ell_r}$) is entire (resp. meromorphic)
admitting a countable set of isolated zeroes (resp. poles) on the complex plane.
\end{enumerate}
\item One has $E\cap L=\{0\}$ and, more generally, $\C[E]\cap\C[L]=\C.1_{\calH(\Omega)}$.
\end{enumerate}
\end{theorem}

\begin{proposition}[\cite{CM}]\label{reg_anal}
Using notations in \eqref{LimoinsHmoins}--\eqref{bullet} and Proposition \ref{explicit}, one has
\begin{enumerate}
\item For any $w\in Y^*$, there is an unique polynomial $p\in\Z[t]$ of valuation $1$ and of degree $(w)+\abs{w}$ such that $R_w=\check p(x_1^*)$ and
\begin{eqnarray*}
\mathrm{f.p.}_{z\rightarrow1}\Li_{R_w}(z)=p(1)\in\Z,
&&\{(1-z)^{a}\log^b((1-z)^{-1})\}_{a\in\Z,b\in\N},\cr
\mathrm{f.p.}_{n\rightarrow+\infty}\H_{\pi_Y(R_w)}(n)=\hat p(1)\in\Q,
&&\{n^{a}\log^b(n)\}_{a\in\Z,b\in\N}.
\end{eqnarray*} 

\item Let $Q\in(\Z[x_0^*,(-x_0)^*,x_1^*],\shuffle,1_{X^*})/(x_0^*\shuffle {x_{1}}^*-{x_1}^*+1)$. Then $\Li_{Q}$ and $\P_{Q}:=\Li_{x_1^*\shuffle Q}\in\Z[z,z^{-1},e^{-\log(1-z)}]$ (and conversely). Moreover,
\begin{eqnarray*}
\mathrm{f.p.}_{z\rightarrow1}\P_{Q}(z)=&\mathrm{f.p.}_{z\rightarrow1}\Li_{Q}(z)\in\Z,& 
\{(1-z)^{a}\log^b((1-z)^{-1})\}_{a\in\Z,b\in\N},\cr
\mathrm{f.p.}_{n\rightarrow+\infty}\scal{\P_Q}{z^n}=&\mathrm{f.p.}_{n\rightarrow+\infty}\H_{\pi_Y(Q)}(n)\in\Q,&\{n^{a}\log^b(n)\}_{a\in\Z,b\in\N}.
\end{eqnarray*}
\end{enumerate}
\end{proposition}

Let $\Upsilon\in\serie{\Q[(n)_{\bullet}]}{Y}$ (resp. $\Lambda\in\serie{\Q[e^{-\log(1-z)}][\log(z)]}{X}$ be the generating series of $\{\H_{\pi_Y(R_w)}\}_{w\in Y^*}$ (resp.
$\{\Li_{R_{\pi_Y(w)}}\}_{w\in X^*}$), with $\scal{\Lambda(z)}{x_0}=\log(z)$. Let $Z^-_{\gamma}\in\serie{\Q}{Y}$ (resp. $Z^-_{\shuffle}\in\serie{\Z}{X}$) be the generating series of $\{\gamma_{\pi_Y(R_w)}\}_{w\in Y^*}$ (resp. $\{\zeta_{\shuffle}(R_{\pi_Y(w)})\}_{w\in X^*}$), with $\scal{Z^-_{\shuffle}}{x_1}=\scal{Z^-_{\shuffle}}{x_0}=0$ and $\scal{Z^-_{\gamma}}{y_1}=-1/2$:
\begin{eqnarray}\label{definition}
\Upsilon:=\sum_{w\in Y^*}\H_{\pi_Y(R_w)}w
&\mbox{and}&
\Lambda:=\sum_{w\in X^*}\Li_{R_{\pi_Y(w)}}w,\\
Z^-_{\gamma}:=\sum_{w\in Y^*}\gamma_{\pi_Y(R_w)}w
&\mbox{and}&
Z^-_{\shuffle}:=\sum_{w\in X^*}\zeta_{\shuffle}(R_{\pi_Y(w)})w.
\end{eqnarray}
Since the morphism $R_{\bullet}:({\C[x_0]}\langle{Y_0}\rangle,{\top},1_{Y_0^*})\longrightarrow(\C[x_0][x_1^*],\shuffle,1_{X^*})$ is bijective then for $\hat{\pi}_Y$ is the morphism of algebras defined by $\hat{\pi}_YS_l=\pi_YS_l$ and $\hat{\pi}_Y(x_0)=x_0$ ($l\in\Lyn X-\{x_0\}$) such that $\Li_{R_{\hat{\pi}_Yx_0}}(z)=\log(z)$ and then $\zeta(R_{\hat{\pi}_Yx_0})=0$, one has
\begin{eqnarray}
\Upsilon=((\H_{\bullet}\circ{\pi}_Y\circ R_{\bullet})\otimes\mathrm{Id}){{\mathcal D}_Y}&\mbox{and}&
\Lambda=((\Li_{\bullet}\circ R_{\bullet}\circ\hat{\pi}_Y)\otimes\mathrm{Id}){{\mathcal D}_X},\\
Z^-_{\gamma}=((\gamma_{\bullet}\circ{\pi}_Y\circ R_{\bullet})\otimes\mathrm{Id}){{\mathcal D}_Y}&\mbox{and}&
Z^-_{\shuffle}=((\zeta_{\shuffle}\circ R_{\bullet}\circ\hat{\pi}_Y)\otimes\mathrm{Id}){{\mathcal D}_X}.
\end{eqnarray}

\begin{theorem}[\cite{CM}]\label{fin}
\begin{enumerate}
\item Associating $l\in\Lyn Y$ with $(s_1,\ldots,s_r)\in\N_{\ge1}^r$, there is a unique $p\in\Z[t]$ of valuation $1$ and of degree $(l)+\abs{l}$ such that 
$$\begin{array}{@{}rcll@{}}
\check p(x_1^*)&=&R_l&\in(\Z[x_1^*],\shuffle,1_{X^*}),\cr
p(e^{-\log(1-z)})&=&\Li_{R_l}(z)&\in(\Z[e^{-\log(1-z)}],\times,1_{\Omega}),\cr
\hat p((n+\bullet)_n)&=&\H_{\pi_Y(R_l)}(n)&\in(\Q[(n+\bullet)_n],\times,1),\cr
\zeta(-s_1,\ldots,-s_r)=p(1)&=&\zeta_{\shuffle}(R_l)&\in(\Z,\times,1),\cr
\gamma_{-s_1,\ldots,-s_r}=\hat p(1)&=&\gamma_{\pi_Y(R_l)}&\in(\Q,\times,1).
\end{array}$$
        
\item $Z^-_{\gamma}$ and $\Upsilon$ (resp. $Z^-_{\shuffle}$ and $\Lambda$) are grouplike for $\Delta_{\stuffle}$ (resp. $\Delta_{\shuffle}$). Moreover,
\begin{eqnarray*}
Z^-_{\gamma}=\prod_{l\in\Lyn Y}^{\searrow}e^{\gamma_{\pi_Y(R_{\Sigma_l})}\Pi_l}    
&\text{and}&
Z^-_{\shuffle}=\prod_{l\in\Lyn X}^{\searrow}e^{\zeta_{\shuffle}(\pi_Y(S_l))P_l},\cr
\Upsilon=\prod_{l\in\Lyn Y}^{\searrow}e^{\H_{\pi_Y(R_{\Sigma_l})}\Pi_l}
&\text{and}&
\Lambda=\prod_{l\in\Lyn X}^{\searrow}e^{\Li_{R_{\pi_Y(S_l)}}P_l}.        
\end{eqnarray*}
        
\item Under the action of $\calG$, for any $g\in\calG$, there exists a substitution $\sigma_g$ and $C\in\ncs{\calL ie_{\C}}{X}$ such that $\Lambda(g(z))=\sigma_g(\Lambda(z))e^C$ and $\Lambda(z)\sim_0e^{x_0\log(z)}$.
\end{enumerate}
\end{theorem}

\begin{thebibliography}{9}
\bibitem{abe}{E. Abe}.--
\textit{Hopf Algebras}, Cambridge Tracts in Mathematics, 2004.

\bibitem{berstel}{J.~Berstel, C.~Reutenauer}.--
\textit{Rational series and their languages}, Spr.-Ver., 1988.
%
\bibitem{B_Alg1}{N. Bourbaki}.--
\textit{Algebra I-III}, Springer-Verlag Berlin and Heidelberg GmbH \& Co. K; (2nd printing 1989)
%
\bibitem{Lie7}{N. Bourbaki}.--
\textit{Groupes et alg\`ebres de Lie, Ch 7-8}, N. Bourbaki et Springer-Verlag Berlin Heidelberg 2006
%
\bibitem{SLC74}{V.C. Bui, G.H.E. Duchamp, V. Hoang Ngoc Minh, Q.H. Ngo and C. Tollu}.--
\textit{(Pure) Transcendence Bases in $\varphi$-Deformed Shuffle Bialgebras},
74\`eme S\'em. Lotharingien de Comb., Haus Sch\"onenberg, Ellwangen (2018).
%
\bibitem{PMB}{V.C. Bui, V. Hoang Ngoc Minh, Q.H. Ngo and V. Nguyen Dinh}.--
\textit{Families of eulerian functions involved in regularization of divergent polyzetas}, Publications Math\'ematiques de Besan\c{c}on, pp. 5-28 (2023).
%
\bibitem{QTS12}{V.C. Bui, V. Hoang Ngoc Minh, Q.H. Ngo and V. Nguyen Dinh}.--
\textit{On the solutions of universal differential equations},
Journal of Physics: Conference Series 2667 (2023).
%
\bibitem{Cartier2}{P. Cartier}.--
\textit{A primer of Hopf algebras},
Frontiers in Number Theory, Physics, and Geometry II,
Springer, Berlin, Heidelberg, pp. 537-615 (2007).
%
\bibitem{CartierPatras}{P. Cartier and F. Patras}.--
\textit{Classical Hopf Algebras and Their Applications}, Algebra and Applications
Springer, (2021).
%
\bibitem{chomskyschutz}{N. Chomsky and M.P. Sch\"utzenberger}.--
\textit{The algebraic theory of context-free langage in computer programing and formal systems},
North Holland Publishing Company, Amsterda (1963).
%
\bibitem{chen}{K.T. Chen}.--
\textit{ Integration of paths-a faithful representation of paths by noncommutative formal power series}, Trans. Amer. Math. Soc. 89 (1958), 395-407.
%
\bibitem{Daresbury}{C.~Costermans and V.~Hoang Ngoc Minh}.--
\textit{Some Results \`a l'Abel Obtained by Use of Techniques \`a la Hopf},
Workshop on Global Integrability of Field Theories and Applications, Daresbury (UK), 1-3, November 2006.
%
\bibitem{JSC}{C. Costermans and V.~Hoang Ngoc Minh}.--
\textit{Noncommutative algebra, multiple harmonic sums and applications in discrete probability}, J. of Sym. Comp. (2009), 801-817.
%
\bibitem{dixmier}{J. Dixmier}.--
\textit{Enveloping algebras}, North-Holland Publishing Company (1977)
%
\bibitem{Linz}{M. Deneufch\^atel, G.H.E. Duchamp, Hoang Ngoc Minh, A.I. Solomon}.--
\textit{Independence of hyperlogarithms over function fields via algebraic combinatorics},
in Lecture Notes in Computer Science (2011), Volume 6742/2011, 127-139.
%
\bibitem{DR}{G. Duchamp, C. Reutenauer}.--
\textit{Un crit\`ere de rationalit\'e provenant de la g\'eom\'etrie non-commutative},
Inventiones Mathematicae, 128, 613-622, (1997)
%
\bibitem{DFLL}{G. Duchamp, M. Flouret, E. Laugerotte, J. G. Luque}.--
\textit{Direct and dual laws for automata with multiplicities}, Theoretical Computer Science, 267, 105-120, (2001)
%
\bibitem{SDSC}{G. H. E. Duchamp and C. Tollu}.--
\textit{Sweedler's duals and Sch\"utzenberger's calculus},
In Combinatorics and Physics, p. 67-78, Amer. Math. Soc. (Contemporary Math., vol.~539), 2011.
%
\bibitem{Ngo}{G.H.E. Duchamp, V. Hoang Ngoc Minh, Q.H. Ngo}.--
\textit{Harmonic sums and polylogarithms at negative multi-indices},
J. of Sym. Comp., { 83 (2017), 166-186.}
%
\bibitem{ACA}{G. Duchamp, V. Hoang Ngoc Minh, Q.H. Ngo, K. Penson, P. Simonnet}.--
\textit{Mathematical renormalization in quantum electrodynamics via noncommutative generating series},
in Applications of Computer Algebra, Springer Proceedings in Math. and Stat., pp. 59-100 (2017).
%
\bibitem{Ngo2}{G.H.E. Duchamp, V. Hoang Ngoc Minh, Q.H. Ngo}.--
\textit{Kleene stars of the plane, polylogarithms and symmetries},
Theoretical Computer Science, 800, p. 52-72, 2019.
%
\bibitem{DGM}{G. Duchamp, D. Grinberg, V. Hoang Ngoc Minh}.--
\textit{Bialgebraic generalizations of linear independence of characters}, In preparation.
%
\bibitem{eilenberg}{S.~Eilenberg}.--
\textit{Automata, Langages and Machines}, Academic press, 1976.
%
\bibitem{siblings}{Enjalbert, Duchamp, Hoang Ngoc Minh, Tollu}--
\textit{The contrivances of shuffle products and their siblings},
Discrete Mathematics 340(9): 2286-2300 (2017).
%
\bibitem{fliess-these}{M.~Fliess}.--
\textit{Sur certaines familles de s\'eries formelles},
Th\`ese d'Etat, Univ. Paris VII, (1972).
%
\bibitem{fliess1}{M.~Fliess}.--
\textit{Fonctionnelles causales non lin\'eaires et ind\'e\-ter\-mi\-n\'ees non commutatives},
Bull. Soc. Math. France, N$^{\circ}$109, 1981, pp. 3-40.
%
\bibitem{fliess2}{M.~Fliess}.--
\textit{R\'ealisation locale des sys\-t\`e\-mes non li\-n\'e\-aires, al\-g\`e\-bres de Lie fil\-tr\'ees transitives et s\'e\-ries g\'en\'e\-ratrices}, Invent. Math., t 71, 1983, pp. 521-537.
%
\bibitem{Hain}{R. Hain}.--
\textit{Iterated integrals and mixed Hodge structures on homotopy groups}, Lecture Notes
in Math., 1246, Springer, Berlin, 1987, 7583.
%
\bibitem{hoangjacoboussous}{V. Hoang Ngoc Minh, G.~Jacob, N.~Oussous}.--
\textit{Input/Output Behaviour of Nonlinear Control Systems~: Rational Approximations, Nilpotent structural Approximations},
in Analysis of controlled Dynamical Systems,
Progress in Sys.s and Cont. Th., Birkh\"auser, 1991, pp. 253-262.
%
\bibitem{FPSAC96}{V. Hoang Ngoc Minh, G. Jacob}.--
\textit{Symbolic Integration of meromorphic differential equation via Dirichlet functions}, Disc. Math. 210, pp. 87-116, 2000.
%
\bibitem{FPSAC98}{V. Hoang Ngoc Minh, M. Petitot, J. Van der Hoeven}.--
\textit{Polylogarithms and Shuffle Algebra},
\textit{Proceedings of FPSAC'98}, 1998.
%
\bibitem{orlando}{V. Hoang Ngoc Minh}.--
\textit{Differential Galois groups and non commutative generating series of polylogarithms}, in ``Automata, Combinatorics and Geometry'', 7th World Multi-conference on Systemics, Cybernetics and Informatics, Florida, USA, Juillet 2003.
%
\bibitem{galoisdiff}{V. Hoang Ngoc Minh}.--
\textit{Shuffle algebra and differential Galois group of colored polylogarithms},
Nuclear Physics B (Proc. Suppl.), 135 (2004), pp. 220-224.
%
\bibitem{livre}{V. Hoang Ngoc Minh}.--
\textit{Calcul symbolique non commutatif}, Presse Acad\'emique Francophone, Saar\-br\"uc\-ken (2014).
%
\bibitem{VJM}{V. Hoang Ngoc Minh}.--
\textit{Structure of polyzetas and Lyndon words}, 
Vietnamese Math. J. (2013),  41, Issue 4, 409-450.
%
\bibitem{CM}{V. Hoang Ngoc Minh}.--
\textit{On the solutions of universal differential equation with three singularit}, in Confluentes Mathematici, Tome 11 (2019) no. 2, p. 25-64.

\bibitem{tokyo}{V. Hoang Ngoc Minh}.--
\textit{Algebraic (Non) Relations Among Polyzetas},
{SCSS 2024: 10th Int. Symp. on Sym. Comp. in Soft. Sci., Aug. 28–30, 2024, Tokyo, Japan.}

\bibitem{Hochschild}{G.P. Hochschild}.--
\textit{Basic Theory of Algebraic Groups and Lie Algebras}. Graduate Texts in Mathematics, vol 75. Springer, New York, NY (1981). 
%
\bibitem{jacob-these}{G.~Jacob}.--
\textit{Repr\'esentation et substitutions matricielles dans la th\'eorie alg\'ebrique des transductions}, Th\`ese d'Etat, Univ. Paris VII (1975).
%
\bibitem{jacob}{G. Jacob}.--
{\it R\'ealisation des syst\`emes r\'eguliers (ou bilin\'eaires)
et s\'eries g\'en\'eratrices non commutatives},
dans ``Outils et mod\`eles math\'ematiques pour l'automatique,
l'analyse de syst\`emes et le traitement du signal", CNRS-RCP 567 (1980).
%
\bibitem{jacobreutenauer}{G.~Jacob, C. Reutenauer (Eds)}.--
Special issue, Theoretical Computer Science 79 (1) (1991)
%
\bibitem{jacobIMACS}{G.~Jacob (Ed)}.--
Special issue, Math. and Comp. in Simulation 42(s 4–6):639 · November 1996
%
\bibitem{kleene}{S.C.~Kleene}.--
\textit{Representation of Events in Nerve Nets and Finite Automata},
Automata studies, Princeton, pp. 3--40 (1956).  
%
\bibitem{kuich}{W.~Kuich}.--
\textit{An Algebraic Characterization of Some Principal Regulated Rational Cones},
J. Comp. Syst. Sc., vol 25, 1982, pp. 377-401.

\bibitem{kuichsalomaa}{W.~Kuich and A.~Salomaa}.--
\textit{Semiring, Automata, Languages}, Springer-Verlag, 2012.
%
\bibitem{lothaire}{M.~Lothaire}.--
\textit{Combinatorics on Words}, Encyclopedia of Mathematics and its Applications, Addison-Wesley, 1983.  
%
\bibitem{MilnorMoore}{J. W. Milnor, J. C. Moore}.--
\textit{On the structure of Hopf algebras}, Ann. of Math. (2) 81 (1965).
%
\bibitem{nivat}{M. Nivat}.--
\textit{Transductions des langages de Chomsky}, Ann. Ins. Fourier (1968) pp. 448-483.
%
\bibitem{reutenauer-these}{C. Reutenauer}.--
\textit{S\'eries rationnelles et Alg\`ebres syntaxiques}, Th\`ese d'Etat, Univ. Paris VII (1980).
%
\bibitem{reutenauerrealisation}{C.~Reutenauer}.--
\textit{The local realisation of generating series of finite Lie rank}.
Algebraic and Geometric Methods In Nonlinear Control Theory. 33-43, (1988).
%
\bibitem{reutenauer}{C. Reutenauer}.--
\textit{Free Lie Algebras}, London Math. Soc. Monographs (1993).
%
\bibitem{SalomaaSoittola}{A. Salomaa and M. Soittola}.--
\textit{Automata-Theoretic Aspects of formal power series}, Spriger-Verlag (1978).
%
\bibitem{schutz2}{M.P. Sch\"utzenberger}.--
\textit{Un probl\`eme de la th\'eorie des automates},
S\'eminaire Dubreil-Pisot, 13th year (1959-1960), no 3, I.H.P., Paris (1960).
%
\bibitem{schutz3}{M.P. Sch\"utzenberger}.--
\textit{On the Definition ofFamily of Automata}, Infor. Contr., Vol 4, pp. 245--270 (1961).
%
\bibitem{schutz4}{M.P. Sch\"utzenberger}.--
\textit{Finite Counting Automata}, Infor. Contr., Vol 4, pp. 91--107 (1962).
%
\bibitem{serre}{J.P. Serre}.--
\textit{Lie Algebras and Lie Groups},
1964 Lectures given at Harvard University, Lecture Notes in Mathematics, vol. 1500, Springer.
%
\bibitem{stanley}{R. Stanley}.--
\textit{Enumerative Combinatorics},
Cambridge University Press, Vol. 1 \& 2 (2011).
%
\bibitem{Sussmann}{H.~Sussmann}.--
\textit{A product expansion of the Chen series},
in Theory and Applications of Nonlinear Control , C.I. Byrnes and A. Lindquist (editors), 1986.
%
\bibitem{viennotgerard}{G.~Viennot}.--
\textit{Alg\`ebres de Lie libres et mono\"\i des libres},
Lecture Notes in Mathematics, Springer-Verlag, 691, 1978.Systems
%
\bibitem{zygmund}{A.~Zygmund}.-- \textit{Trigonometric series}, Cambridge University Press (2002).
\end{thebibliography}
\end{document}